\documentclass[preprint]{imsart}

\RequirePackage[OT1]{fontenc}
\RequirePackage{amsfonts}
\RequirePackage[numbers]{natbib}
\RequirePackage[colorlinks,citecolor=blue,urlcolor=blue]{hyperref}
\RequirePackage[centertags, leqno, sumlimits, nointlimits, namelimits]{mathtools}
\RequirePackage[thmmarks,thref,amsthm,amsmath]{ntheorem}
\RequirePackage{paralist}
\RequirePackage[english, noabbrev]{cleveref}

\startlocaldefs
\numberwithin{equation}{section}
\theoremstyle{plain}
\newtheorem{theorem}{Theorem}[section]
\newtheorem{lemma}[theorem]{Lemma}
\newtheorem{corollary}[theorem]{Corollary}
\newtheorem{prop}[theorem]{Proposition}

\theoremstyle{remark}
\newtheorem{remark}[theorem]{Remark}

\endlocaldefs

\newcommand{\N}{\mathbb{N}} 
\newcommand{\R}{\mathbb{R}} 
\newcommand{\ew}{\mbox{I\negthinspace E}} 
\newcommand{\lebesgue}{\ensuremath{\lambda\!\!\!\;\!\lambda}} 
\begin{document}

\begin{frontmatter}
\title{Detectability of nonparametric signals: higher criticism versus likelihood ratio}
\runtitle{Detectability of nonparametric signals}

\begin{aug}
\author{\fnms{Marc} \snm{Ditzhaus}\thanksref{t2}\ead[label=e1]{marc.ditzhaus@hhu.de}}
\address{Institute of Statistic, Ulm University\\
	 Helmholtzstra{\ss}e 20,  89081 Ulm, Germany\\
	\printead{e1}
}

\and
\author{\fnms{Arnold} \snm{Janssen}\thanksref{t2}\ead[label=e2]{janssena@uni-duesseldorf.de}}
\address{Mathematical Institute, Heinrich-Heine University D\"usseldorf\\
	Universit\"atststra{\ss}e 1, 40225 D\"usseldorf, Germany\\
	\printead{e2}
}

\thankstext{t2}{Supported by DFG Grant no. 618886}
\runauthor{M. Ditzhaus and A. Janssen}

\end{aug}

\begin{abstract}
	We study the signal detection problem in high dimensional noise data (possibly) containing rare and weak signals. Log-likelihood ratio (LLR) tests depend on unknown parameters, but they are needed to judge the quality of detection tests since they determine the detection regions.  The popular Tukey's higher criticism (HC) test was shown to achieve the same completely detectable region as the LLR test does for different (mainly) parametric models. We present a novel technique to prove this result for very general signal models, including even nonparametric $p$-value models. Moreover, we address the following questions which are still pending since the initial paper of Donoho and Jin: \textit{What happens on the border of the completely detectable region, the so-called detection boundary? Does HC keep its optimality there?} In particular, we give a complete answer for the heteroscedastic normal mixture model. As a byproduct, we give some new insights about the LLR test's behaviour on the detection boundary by discussing, among others, Pitmans's asymptotic efficiency as an application of Le Cam's theory. 
\end{abstract}



\end{frontmatter}
\section{Introduction}\label{sec:intro}
Signal detection in huge data sets becomes more and more important in current research. The number of relevant information is often a quite small part of the data set and hidden there. In genomics, for example, the assumption is often used that the major part of the genes in patients affected by some common diseases like cancer behaves like white noise and a minor part is differentially expressed but only slightly, see \cite{DaiETAL2012,Goldstein2009,IyengarElston2007}. Consequently, the number of signals as well as the signal strength is small. This circumstance makes it difficult to decide whether there are any signals. Other application fields are disease surveillance, see \cite{KulldorffETAL2005,NeillLingwall2007}, local anomaly detection, see \cite{SaligramaZhao2012}, cosmology and astronomy, see \cite{CayonETAL2004,JinETAL2005Cosmo}. In the last decade \textit{Tukey's higher criticism} (HC) test, see \cite{TukeyCoursenotes,TukeyInternalPaper,TukeyCollected}, modified by  \citet{DonohoJin2004} became quite popular for these kind of problems. The reason for HC's popularity is that the area of complete detection coincide for the HC test and the \textit{log-likelihood ratio} (LLR) test under different specific model assumption, see \cite{AriasWang2015,AriasWang2017,CaiJengJin2011,Cai_Wu_2014,DonohoJin2004,Jin2004}. This was also done for  sparse linear regression models and binary regression models, see \cite{AriasCandesPlan2015,IngsterETAL2010,MukherjeePillaiLin2015}. To overcome the problem of an unknown noise distribution \cite{DelaigleETAL2011} used a bootstrap version of HC. A lot of related literature to the possibilities of HC, even beyond signal detection, can be found in the survey paper of \cite{DonohoJin2015}. For instance, \cite{HallETAL2008} applied HC for classification. \\
There are (only) a few results concerning the asymptotic power behaviour of the LLR test on the detection boundary, which separates the area of complete detection and the area of no possible detection, see e.g. \cite{CaiJengJin2011,Ingster1997}  for the heteroscedastic and heterogeneous normal mixture models. Since \cite{DonohoJin2004} the questions is pending: \textit{How does HC perform on the detection boundary? Does it keep its optimality?} \cite{DonohoJin2004} specially pointed out this question: "\textit{Just at the critical point where r = ρ ∗ (1 + o(1)), our result says nothing; this would be an interesting (but very challenging) area for future work.}" \\
Our paper's purpose is twofold. First, we want to fill the theoretical gap concerning the tests' power behaviour on the detection boundary and give an answer to the question mentioned before. We quantify the asymptotic power of the LLR test by giving the LLR statistic's limit distribution. On the detection boundary the LLR test has nontrivial asymptotic power, whereas the HC test does not. Consequently, HC is not overall powerful. However, our message is not to scrap the idea of HC. Its power behaviour is still optimal beyond the detection boundary for a long list of models. The second purpose of our paper is to add a $p$-value model with signals coming from a nonparametric alternative to this list of models. \\
The paper is organized as follows. In \Cref{sec:model} we introduce the general model and the detection testing problem. For the readers' convenience the context and the main results are briefly illustrated  for a (specific) nonparametric model in \Cref{sec:intro_h}. 
The asymptotic results about the LLR test appear  \Cref{sec:LLRT}. \Cref{sec:HCT} is devoted to the HC statistic and introduce an "HC complete detection" as well as a "trivial HC power" Theorem. \Cref{sec:applications} contains the application of our theory. There we discuss a generalisations of the illustrative results from \Cref{sec:intro_h} and the heteroscedastic normal mixture model. Although the latter was already studied in great detail we can give some new insights for it. All proofs are relegated to \Cref{appendix:proofs}

\subsection{The model}\label{sec:model}
Let $\{k_n:n\in\N\}\subset \N$, where $k_n\to\infty$ represents the number of observations. Throughout this paper, if not stated otherwise all limits are meant as $n\to\infty$. Let the following three mutually independent triangular arrays consisting of rowwise independent random variables are given, where values in different spaces are allowed: 
\begin{itemize}
\item $(Z_{n,i})_{i\leq k_n}$ representing the noisy background, where the distribution $P_{n,i}$ of $Z_{n,i}$ is assumed to be known. In the applications we often assume that $P_{n,i}=P_0$ depends neither on $i$ nor on $n$, and $P_0$ may stand for a distribution of $p$-values under the null.

\item $(X_{n,i})_{i\leq k_n}$ representing the signals, where the signal distribution $\mu_{n,i}$ of $X_{n,i}$ is typically unknown.

\item $(B_{n,i})_{i\leq k_n}$ representing the appearance of a signal, where $B_{n,i}$ is Bernoulli distributed with typically unknown success probability $0\leq \varepsilon_{n,i}\leq 1$. 
\end{itemize}
Instead of these random variables we observe 
\begin{align*}
Y_{n,i}\;=\; \begin{cases}
X_{n,i} & \;\textrm{ if }B_{n,i}=1 \\
Z_{n,i} & \;\textrm{ if }B_{n,i}=0 \\
\end{cases}
\end{align*}
for all $1\leq i \leq k_n$. The vector $(Y_{n,1},\ldots,Y_{n,k_n})$ represents the noise data containing a random amount $\sum_{i=1}^{k_n} B_{n,i}$ of signals. It is easy to check that the distribution $Q_{n,i}$, say, of $Y_{n,i}$ is given by
\begin{align}\label{eqn:def_qni}
Q_{n,i}= (1-\varepsilon_{n,i})P_{n,i}+\varepsilon_{n,i}\mu_{n,i}=P_{n,i}+\varepsilon_{n,i}(\mu_{n,i}-P_{n,i}).
\end{align}
We are interested whether there are any signals in the noise data, i.e. whether $B_{n,i}=1$ for at least one $i=1,\ldots,k_n$. To be more specific, we study the testing problem 
\begin{align}\label{eqn:testing_problem}
{\mathcal H }_{0,n}:\; \varepsilon_{n,i}=0\text{ for all i}\qquad\textrm{versus}\qquad {\mathcal H }_{1,n}:\;\varepsilon_{n,i}>0 \text{ for at least one }i,
\end{align}
where we observe pure noise $(Y_{n,1},\ldots,Y_{n,k_n})=(X_{n,1},\ldots,X_{n,k_n})$ under the null.
We are especially interested in the case of rare signals in the sense that
\begin{align}\label{eqn:maxeps_to_0}
\max_{1\leq i \leq k_n} \varepsilon_{n,i} \to 0.
\end{align} 
Another typical assumption in the signal detection literature is
\begin{align}\label{eqn:absolute_conti}
\mu_{n,i}\ll P_{n,i}\textrm{ for all }1\leq i \leq k_n,
\end{align}
which we also suppose throughout this paper. In \Cref{sec:violation} we discuss what happens if the assumption of absolute continuity is violated. Following the ideas of \citet{Cai_Wu_2014} we explain that every model can be reduced to a model such that \eqref{eqn:absolute_conti} is fulfilled.\\
\textbf{Convention and Notation:} Observe that 
\begin{align*}
\frac{ \,\mathrm{ d } Q_{n,i}}{\,\mathrm{ d } P_{n,i}} = 1 + \varepsilon_{n,i}\Bigl( \frac{ \,\mathrm{ d } \mu_{n,i}}{\,\mathrm{ d } P_{n,i}} -1\Bigr).
\end{align*}
The distributions $P_{n,i},\mu_{n,i},Q_{n,i}$ and the densities $\frac{\mathrm{ d } Q_{n,i}}{\,\mathrm{ d }P_{n,i}}\circ pr_i$ shall lie on the same product space, where the projections $pr_i$ on the \textit{i}th coordinate are suppressed throughout the paper to improve  the readability. Moreover, we introduce the product measures
\begin{align*}
Q_{(n)}=\bigotimes_{i=1}^{k_n} Q_{n,i}\textrm{ and }P_{(n)}=\bigotimes_{i=1}^{k_n} P_{n,i}.
\end{align*} 

\subsection{Illustration of the results and the main contents}\label{sec:intro_h}
Here, our results are briefly presented for a special nonparametric $p$-values model. For simplicity let $k_n=n$. Since we are dealing with $p$-values the null (noise) distribution is the uniform distribution on $(0,1)$, i.e. $P_{n,i}=P_0=\lebesgue_{|(0,1)}$ for all $1\leq i \leq n.$ Note that as long as the noise distribution is continuous this is not a restriction having a quantile transformation $P_{n,i}((Y_{n,i},\infty))$ or $P_{n,i}((-\infty,Y_{n,i}])$ in mind. Typically, small p-values indicates that the alternative is true, or in our case that signals are present. Respecting this we suggest signal distributions $\mu_{n,i}$ with a shrinking support $[0,\kappa_n]$, where
\begin{align}\label{eqn:intro_h_tau+eps}
\kappa_n=n^{-r}\textrm{ and }\varepsilon_{n,i}=\varepsilon_n=n^{-\beta}
\end{align}
for some $\beta\in(1/2,1)$ and $r>0$. In order to obtain such a distribution $\mu_{n,i}$ the interval $(0,\kappa_n)$ is blown up to $(0,1)$ and a nonparametric shape function $h$ is used. Let $h:(0,1)\to (0,\infty)$ be a Lebesgue probability density, i.e. $\int_0^1 h \,\mathrm{ d } \lebesgue =1$, with $\int_0^1 h^2 \,\mathrm{ d }\lebesgue\in(0,\infty)$ and define the signal distribution $\mu_{n,i}=\mu_{n}$ by its rescaled Lebesgue density 
\begin{align}\label{eqn:ill_density}
\frac{\mathrm{ d } \mu_{n}}{\,\mathrm{ d }\lebesgue_{|(0,1)}}(x) =  \frac{1}{\kappa_{n}} h\Bigl( \frac{x}{\kappa_{n}} \Bigr) \mathbf{1}\{x\leq \kappa_{n}\},\;x\in (0,1).
\end{align}
Since it could be to restrictive in practice to consider only measures with a shrinking support,  in \Cref{sec:h-model} we add a "small" perturbation to the densities. To sum up, we have a nonparametric testing problem which can be expressed heuristically as
\begin{align*}
\mathcal H_{0,n}: \varepsilon_n=0\text{ versus }\mathcal H_{1,n}: \varepsilon>0, h\in L^2(P_0)\text{ with }h\geq 0,\,\int h \mathrm dP_0=1.
\end{align*}
In the following sections we give answers to the seven problems \ref{enu:ill_exam_determ_detect_bound}-\ref{enu:intro_h_HC_bound} for the general model introduced in \Cref{sec:model} and present here the corresponding results for the illustrative nonparametric $p$-value model.

\begin{figure}[tb] 
\begin{center}	
	\includegraphics[trim = 22mm 150mm 80mm 10mm, clip, width=0.45\textwidth]{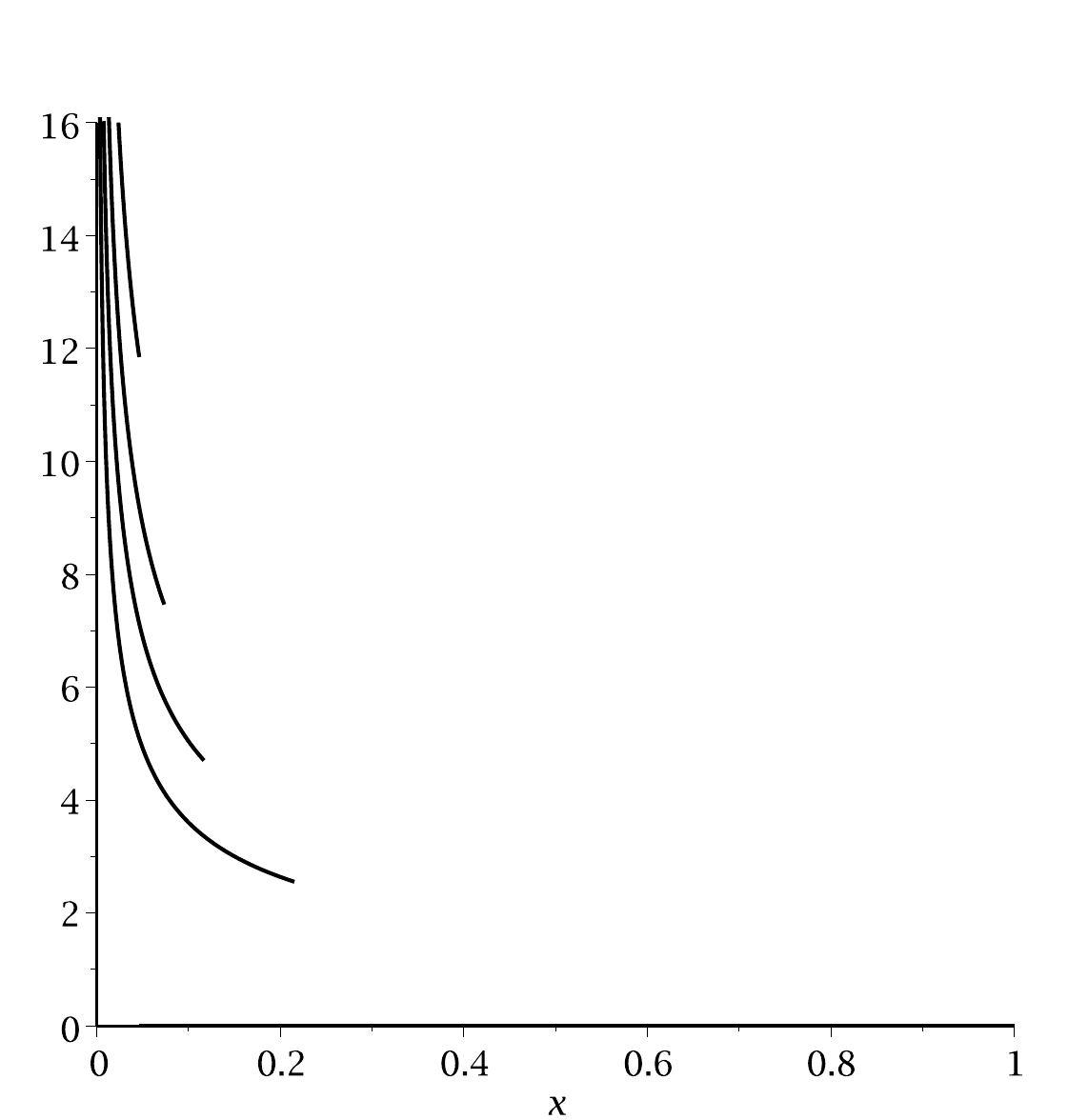}\quad
	\includegraphics[trim = 22mm 150mm 80mm 10mm, clip, width=0.45\textwidth]{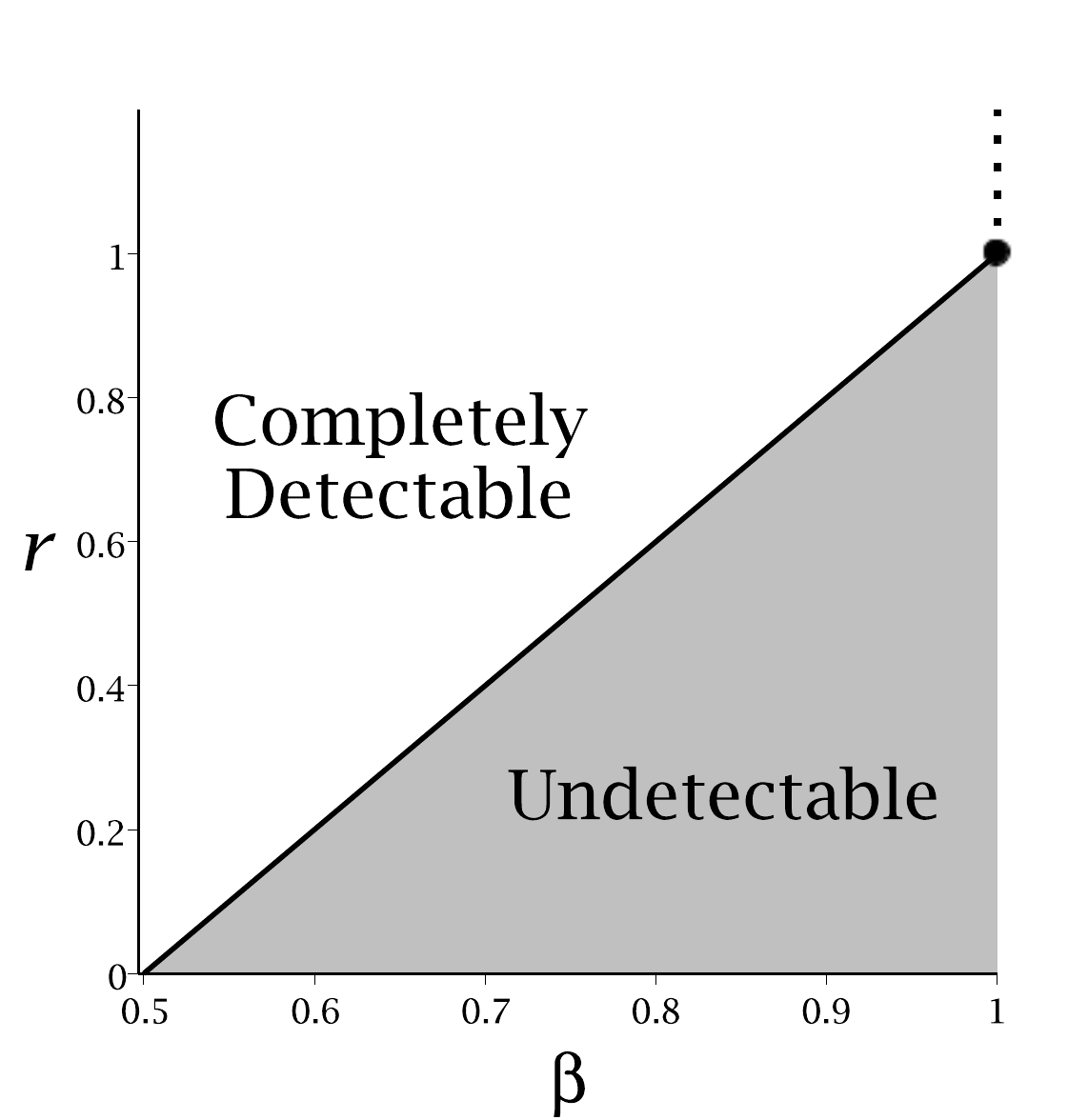}
\end{center}
\caption[Detection boundaries for the sparse and the dense heteroscedastic normal mixture model]{Left: Plot of $x\mapsto \mathrm{ d } \mu_{n}/\mathrm{ d }P_{0}(x)$ for $h(x)=(1-a)x^{-a}$, $a=9/20$, $r=2/3$ and $n\in\{10,25,50,100\}$, see \eqref{eqn:ill_density}. Right: The nonparametric detection boundary is plotted. Above the boundary is the completely detectable area and underneath is the undetectable area. The limits of the LLR statistic are Gaussian on the solid line  under the null as well as under the alternative, and they are real-valued but non-Gaussian on the end of the line (solid circle). The limit under the alternative is equal to $\infty$ with a positive probability.} \label{fig:hmodel}
\end{figure}

\begin{enumerate}[I.]
\item\label{enu:ill_exam_determ_detect_bound} \textit{Determination of the detection boundary}: Since the paper of \citet{DonohoJin2004} the term \textit{detection boundary} is of great interest for the detection problem. This boundary splits the $r$-$\beta$ parametrisation plane into the \textit{completely detectable} and the \textit{undetectable} area. For each pair $(r,\beta)$ from the \textit{completely detectable} area the  LLR test, the optimal test, can completely separate the null and the alternative asymptotically. This means that there is a sequence $(\varphi_n)_{n\in\N}$ of LLR tests with nominal levels $E_{P_{(n)}}(\varphi_n)=\alpha_n$ such that $\alpha_n\to 0$ and the power $E_{Q_{(n)}}(\varphi_n)$ under the alternative tends to $1$. For each $(r,\beta)$ from the undetectable area  the null ${\mathcal H }_{0,n}$ and the alternative ${\mathcal H }_{1,n}$ are asymptotically indistinguishable, i.e. the sum of error probabilities tends to $1$ for each possible sequence of tests. Hence, no test yields asymptotically better results than a constant test $\varphi\equiv \alpha\in(0,1)$. For the illustrative model we have a \textit{nonparametric} detection boundary which is independent of the shape function $h$ and given by
\begin{align}\label{eqn:hmodel_detect_bound}
\rho(\beta)=2\beta - 1\textrm{ for }\beta\in\Bigl( \frac{1}{2},1 \Bigr].
\end{align}
The area where $r>\rho(\beta)$ ($r<\rho(\beta)$, resp.) corresponds to the completely detectable area (undetectable area, respectively), see \Cref{fig:hmodel}.

\item\label{enu:intro_h_LLRT_on_db} \textit{Gaussian limits on the detection boundary?} For some parametric models the limit distribution of the log-likelihood ratio test statistic $T_n$, see below, was determined, e.g. for the heteroscedastic and heterogeneous normal mixture model, see \citet{CaiJengJin2011} and \citet{Ingster1997}. For our model with $1/2<\beta<1$ and $r=\rho(\beta)$ we have
\begin{align*}
T_n=\log \frac{\mathrm{ d } Q_{(n)}}{\,\mathrm{ d }P_{(n)}} \overset{\mathrm d}{\longrightarrow} \left\{\begin{array}{ll}
\xi_1\sim N( -\frac{\sigma^2(h)}{2},\sigma^2(h))&\textrm{ under } {\mathcal H }_{0,n},\\
\xi_2\sim N( \phantom{-}\frac{\sigma^2(h)}{2},\sigma^2(h))&\textrm{ under } {\mathcal H }_{1,n},
\end{array}\right. 
\end{align*}
where $\sigma^2(h)=\int_0^1 h^2 \,\mathrm{ d }\lebesgue.$ Observe that the limits only depend on the second moment of $h$ and not on its specific structure. 

\item \label{enu:intro_hmod_what_wrong_h_beta} \textit{What happens if we choose the wrong $h$ or $\beta$ for the LLR 
	statistic on the boundary?} Let $(h_1,\beta_1)$ and $(h_2,\beta_2)$ represent two specific models of the illustrative example on the detection boundary, i.e. $\beta_i\in(1/2,1)$ and $r_i	=\rho(\beta_i)$ for $i=1,2$. Using Le Cam's LAN theory we can determine the asymptotic power of the LLR test $\varphi_{n,\beta_2,h_2,\alpha}$ of the model $(h_2,\beta_2)$  of nominal level $\alpha\in(0,1)$ if $(h_1,\beta_1)$ is the true, underlying model:
\begin{align*}
&E_{\mathcal H_{1,n}( h_1, \beta_1)}(\varphi_{n,\beta_2,h_2,\alpha}) \to   
\Phi\Bigl( u_{\alpha} + \sqrt{\sigma^2(h_1)\mathrm{ARE}}\Bigr),\\
&\text{where }\mathrm{ARE}=\frac{(\int_0^1 h_1h_2\,\mathrm{ d }\lebesgue )^2}{\sigma^2(h_1)\sigma^2(h_2)}\mathbf{1}\{\beta_1=\beta_2\}
\end{align*}
is Pitman's asymptotic relative efficiency, see \cite{HajekSidakSen}, $\Phi$ denotes the distribution function of a standard normal distribution and $u_\alpha$ is the corresponding $\alpha$-quantile, i.e. $\Phi(u_\alpha)=\alpha$. This formula quantifies the loss of power by choosing the wrong $\beta$ or $h$. In particular, the LLR test $\varphi_{n,\beta_2,h_2,\alpha}$ cannot separate the null and the alternative asymptotically, i.e ARE$=0$, if the supports of $h_1$ and  $h_2$ are disjunct, or if $\beta_1$ and $\beta_2$ are unequal. 

\item\label{enu:intro_h_nonreal_xi2} \textit{Beyond Gaussian limits on the detection boundary.} Non-Gaussian limits of  $T_n$ may occur, see \cite{CaiJengJin2011,Ingster1997}. Here, such limits can be observed for $\beta=1$,  $r=\rho(1)=1$. The limits are infinitely divisible distributed with nontrivial L\'{e}vy measure. These L\'{e}vy measures depend  heavily on the special structure  of $h$. For further results with infinitely divisible non-Gaussian $\xi_1$, $\xi_2$ confer \Cref{lem:h=xrho}, where we investigate in shape functions $h$ with $\int_0^1 h^2\,\mathrm{ d }\lebesgue=\infty$. Beside all this we also observe a new class of limit. To be more specific, the limit of $T_n$ equals $\infty$ with positive probability under the alternative, whereas the limit under the null is always real-valued (except, of course, in the completely detectable case). For $\beta=1$ and $r>1$
\begin{align*}
T_n \overset{\mathrm d}{\longrightarrow} \left\{\begin{array}{ll}
\xi_1\equiv -1 &\textrm{ under } {\mathcal H }_{0,n},\\
\xi_2\sim e^{-1}\epsilon_{-1}+(1-e^{-1})\epsilon_{-\infty}&\textrm{ under } {\mathcal H }_{1,n},
\end{array}\right.  
\end{align*}
where $\epsilon_{a}$ denotes the Dirac measure centered in $a\in[-\infty,\infty]$, i.e. $\epsilon_{a}(A)=\mathbf{1}\{x\in A\}$. 
As far as we know such limits were not observed for the detection issue until now. All statements about $\beta=1$ even hold if $\int_0^1h^2\,\mathrm{ d }\lebesgue=\infty$. 

\item\label{enu:ill_exam_extens_detec_bound}\textit{Extension of the detection boundary:} 
As stated in \eqref{enu:intro_h_nonreal_xi2} our discussion includes $\beta=1$, whereas a lot of former research was focused (only) on $\beta<1$. The case $\beta\geq 1$ was of minor interest reason since the probability that at least one signal is present equals $1-(1- \varepsilon_n)^{n}$, which tends to $1-e^{-1}$ and $0$ if $\beta=1$ and $\beta>1$, respectively. In particular, the pair $(\beta,r)$ with $\beta>1$ and $r>0$ always belongs to the undetectable area. Hence, $\beta>1$ do not need to be studied further. But $\beta=1$ should be taken into account, at least, when nontrivial limits are of the researcher's interest. To sum up, the detection boundary can be extended by the case $\beta=1$, see \Cref{fig:hmodel}. 

\item\label{enu:intro_h_HC_opt}\textit{Optimality of HC.} As already known for different mainly parametric models, we can show also for the illustrative nonparametric $p$-values model that the completely detectable regions of the LLR and the HC test coincide. By this we give a further reason why HC is a good candidate for the signal detection problem.

\item\label{enu:intro_h_HC_bound}\textit{No power of HC on the boundary.} We show that on the detection boundary, i.e. $\beta\in(1/2,1)$ and $r=\rho(\beta)$, the HC test cannot distinguish between the null and the alternative alternative, whereas the LLR test has nontrivial power, compare to \ref{enu:intro_h_LLRT_on_db}. 
\end{enumerate}
Among others, we apply our results to the model \eqref{eqn:ill_density} in a more general form, e.g. $h_{n,i}$, $\kappa_{n,i}$ and $\varepsilon_{n,i}$ may depend on $i$ and $n$.  We want to point out that these kind of alternatives were already studied in the context of goodness-of-fit testing by \citet{Khmaladze1998}. He used the name \textit{spike chimeric alternatives}. Finally, we want to mention that our general model and the upcoming results also include
\begin{itemize}
\item discrete models (only for the LLR test), as the Poisson model of \citet{AriasWang2015}.
\item the \textit{sparse}  ($\sum_{i=1}^{k_n} \varepsilon_{n,i}^2\to 0$), the \textit{classical}  ($\lim_{n\to\infty}\sum_{i=1}^{k_n}\varepsilon_{n,i}^2\in(0,\infty)$) and the \textit{dense}/\textit{moderately sparse} case ($\sum_{i=1}^{k_n} \varepsilon_{n,i}^2\to \infty$). We used the word "classical" for the second case corresponding heuristically to $\varepsilon_{n,i}\approx 1/\sqrt{k_n}$ which is the convergence rate typically used in the context of contiguous alternatives.  
\end{itemize}


\section{Asymptotic power behaviour of LLR tests}\label{sec:LLRT}
In this section we discuss the asymptotic power behaviour of LLR tests. These tests depend on the unknown signals and, hence, they are not applicable. But they serve as an import benchmark and all new suggested tests should be compare with the optimal LLR tests. \\
It is well known that at least for a subsequence $T_n$ converges in distribution to a random variable with values on the extended real line $[-\infty,\infty]$ under the null as well as under the alternative, see Lemma 60.6 of \citet{Strasser1985}.  That is why we can assume without loss of generality that
\begin{align}\label{eqn:def_Tn}
T_n=\sum_{i=1}^{k_n}\log \frac{\mathrm{ d } Q_{n,i}}{\,\mathrm{ d }P_{n,i}}(Y_{n,i}) {\overset{\mathrm d}{\longrightarrow}}\left\{ \begin{array}{ll}
\xi_1\textrm{ under }P_{(n)}\text{ (null)},\\
\xi_2\textrm{ under }Q_{(n)}\text{ (alternative)},
\end{array}\right. 
\end{align}
where $\xi_1$ and $\xi_2$ are random variables on $[-\infty,\infty]$. Regarding the phase diagram on the right side in \Cref{fig:hmodel} we are interested in the following three different regions/cases:
\begin{enumerate}[(i)]
\item (Completely detectable) The LLR test $\varphi_n=\mathbf{1}\{T_n>c_n\}$ with appropriate critical values $c_n\in\R$ can completely separate the null and the alternative asymptotically, i.e. the sum of error probabilities $E_{\mathcal H_{0,n}}(\varphi_n)+ E_{\mathcal H_{1,n}}(1-\varphi_n)$ tends to $0$. We will see that this corresponds to $\xi_1\equiv -\infty$ and $\xi_2\equiv \infty$. 

\item (Undetectable) No test sequence $(\psi_n)_{n\in\N}$ can distinguish between the null and the alternative asymptotically, i.e we always have $E_{\mathcal H_{0,n}}(\varphi_n)+ E_{\mathcal H_{1,n}}(1-\varphi_n)\to 1$. This case corresponds to $\xi_1\equiv 0\equiv \xi_2$. 

\item (Detectable) The LLR test $\varphi_n=\mathbf{1}\{T_n>c_n\}$ with appropriate critical values $c_n\in\R$ can separate the null and the alternative asymptotically but not completely, i.e. $E_{\mathcal H_{0,n}}(\varphi_n)+ E_{\mathcal H_{1,n}}(1-\varphi_n)\to c\in (0,1)$.
\end{enumerate}
In the following we denote the completely detectable and the undetectable case as the trivial cases since the limits of $T_n$ are degenerated. We start by discussing these and we present a useful tool to verify these trivial cases/limits of $T_n$. After that we will see that the same tools can be used to determine the nontrivial limits in the detectable case. In the last two subsections we consider the asymptotic relative efficiency, compare to \eqref{enu:intro_hmod_what_wrong_h_beta} from \Cref{sec:intro_h}, and explain what to do when the condition \eqref{eqn:absolute_conti} is violated.

\subsection{Trivial limits}\label{sec:trivial_limits}
In the proofs we work with different distances for probability measure, among others the Hellinger distance and the variational distance. Using theses distances we can classify the different detection regions. We refer the reader to the \Cref{appendix:proofs}, for further details. Here, we only present our new tool. Let us introduce for all $x>0$ the following two sums 
\begin{align}
&I_{n,1,x}=\sum_{i=1}^{k_n} \varepsilon_{n,i}\mu_{n,i}\Bigl( \varepsilon_{n,i}\frac{\mathrm{ d } \mu_{n,i}}{\,\mathrm{ d }P_{n,i}} > x  \Bigr)\label{eqn:I1}
\\
\textrm{and }&I_{n,2,x}=\sum_{i=1}^{k_n} \varepsilon_{n,i}^2E_{P_{n,i}} \Bigl( \Bigl( \frac{\mathrm{ d } \mu_{n,i}}{\,\mathrm{ d }P_{n,i}} \Bigr)^2 \mathbf{1}\Bigl\{ \varepsilon_{n,i}\frac{\mathrm{ d } \mu_{n,i}}{\,\mathrm{ d }P_{n,i}} \leq x \Bigr\} - 1  \Bigr).	\label{eqn:I2}			
\end{align}
\begin{theorem}\label{theo:trivial_limits}
Let $\tau>0$ be fixed.
\begin{enumerate}[(a)]
	\item\label{enu:lem_trivial_limits_comp_det} The completely detectable case is present  if and only if $I_{n,1,\tau}$ or $I_{n,2,\tau}$ tends to $\infty$.
	
	\item\label{enu:lem_trivial_limits_undet} We are in the undetectable case if and only if $I_{n,1,\tau}$ as well as $I_{n,2,\tau}$ tends to $0$.
\end{enumerate}
\end{theorem}

\subsection{Nontrivial limits}

It turns out that only a special class of distributions $\nu_1$ and $\nu_2$, say, of $\xi_1$ and $\xi_2$ may occur. The results fit in the more general framework of statistical experiments: all nontrivial weak accumulation points with respect to the weak topology of statistical experiments are infinitely divisible statistical experiments in the sense of \citet{LeCam1986}, see \cite{LeCamYang2000} and \cite{Janssen1990}. In the following we explain what this means in our situation. Classical infinitely divisible distributions on $(\R,{\mathcal B })$ play a key role for our setting. That is why we want to recall that the characteristic function $\varphi$ of an infinitely divisible distribution on $(\R,{\mathcal B })$ is given by the L\'{e}vy-Khintchine formula 
\begin{align*}
\varphi(t)\; = \; \exp\Bigl[ \text{i} \gamma t - \frac{\sigma^2 t^2}{2}+\int_{\R\setminus\{0\}}\, \Bigl( \exp(\text{i} t x) - 1 - \frac{\text{i} t x}{1+x^2} \Bigr) \,\,\mathrm{ d } \eta(x) \Bigr],\, t\in\R,
\end{align*}
where $\gamma\in\R$, $\sigma^2\in[0,\infty)$ and $\eta$ is a L\'{e}vy measure, i.e. $\eta$ is a measure on $\R\setminus\{0\}$ with $\int \min(x^2,1) \,\mathrm{ d } \eta <\infty$. The triple $(\gamma,\sigma^2,\eta)$ is called the L\'{e}vy-Khintchine triple and is unique. See \citet{gndedenKolmogorov} for more details about infinitely divisible distributions.
The following theorem gives us a characterisation of all possible limits  of $T_n$.
\begin{theorem}\label{theo:xi_real_oder_xi_fullinfo}
\begin{enumerate}[(a)]		\item\label{enu:theo:xi_real_oder_fullinfo} Either $\xi_1$ is real-valued or $\xi_1\equiv -\infty$ with probability one. In case of the latter $\xi_2\equiv \infty$ with probability one. 
	
	\item\label{enu:theo:xi_real_oder_xi_fullinfo_infinit_divi} Suppose $\xi_1$ is real-valued. Then $a=P(\xi_2\in\R)>0$ and we can rewrite $\nu_2=a\rho+(1-a)\epsilon_{\infty}$, where $\rho(A)=a^{-1}\nu_2(A\cup \R)$ for all $A\in\mathcal B([-\infty,\infty])$. Moreover, $\nu_1$ and $\rho=a^{-1}\nu_{2|\R}$ are infinitely divisible distributions on $(\R,\mathcal B)$. Let $(\gamma_1,\sigma_1^2,\eta_1)$ and $(\gamma_2,\sigma_2^2,\eta_2)$ be the L\'{e}vy-Khintchine triplets of $\nu_1$ and $\rho=a^{-1}\nu_{2|\R}$. Then we have:
	\begin{enumerate}[(i)]
		\item\label{eqn:theo:connect_Levymeasure} The L\'{e}vy measures $\eta_1$ and $\eta_2$ are concentrated on $(0,\infty)$, i.e. $\eta_j(-\infty,0)=0$.  and $\int_{(0,\infty)} e^x \,\mathrm{ d }\eta_1(x)<\infty$. Moreover, $\frac{\mathrm{ d } 			\eta_2}{\,\mathrm{ d }\eta_1}(x)=e^x\textrm{ for all }x>0.$
		
		\item\label{eqn:theo:connect_sigma} The variances of the Gaussian parts of $\xi_1$ and $\xi_2$ coincide, i.e. $\sigma_1^2=\sigma_2^2$.
		
		\item\label{eqn:theo:connect_gamma} The drift parameters $\gamma_1$ and $\gamma_2$ fulfill the formulas:
		\begin{alignat}{2}
		&\log(a)&&=\gamma_1 +\frac{\sigma^2_1}{2}-\int_{(0,\infty)} \Bigl(1-e^x  + \frac{x}{1+x^2} \Bigr) \:\,\mathrm{ d }\eta_1(x),\label{eqn:gamma1}\\
		&\gamma_2&& = \gamma_1 + \sigma^2_1 +  \int_{(0,\infty)} ( e^x-1)\frac{x}{1+x^2} \,\mathrm{ d }\eta_1(x). \label{eqn:gamma_2_general} 	
		\end{alignat}
	\end{enumerate}
\end{enumerate}
\end{theorem}
\begin{remark}\label{rem:contiguous}
If $\xi_1$ is real-valued then by Le Cam's first Lemma the null (product) measure $P_{(n)}$ is contiguous with respect to the alternative (product) measure $Q_{(n)}$, i.e. $Q_{(n)}(A_n)\to 0$ implies $P_{(n)}(A_n)\to 0$. If additionally $\xi_2$ is real-valued then $P_{(n)}$ and $Q_{(n)}$ are mutually contiguous, i.e. $Q_{(n)}(A_n)\to 0$ if and only if $P_{(n)}(A_n)\to 0$. Observe that under mutually contiguity a random variable is asymptotically constant under the null $P_{(n)}$ if and only if this is the case under the alternative $Q_{(n)}$.
\end{remark}
According to \Cref{theo:xi_real_oder_xi_fullinfo}\eqref{enu:theo:xi_real_oder_xi_fullinfo_infinit_divi} the L\'{e}vy-Khintchine triplets of $\nu$ and $\rho=a^{-1}\nu_{2|\R}$ are closely related to each other. This was already observed in the context of statistical experiments by \citet{JanssenMilbrodtStrasser1985}. \\
Now, we know the class of all possible limits and, hence, the questions arises naturally how to determine the distribution of $\xi_1$ and $\xi_2$ for a given setting. To answer this question we first observe that by \Cref{theo:xi_real_oder_xi_fullinfo}\eqref{eqn:theo:connect_Levymeasure} the L\'{e}vy measures $\eta_1$ and $\eta_2$ are uniquely determined by their difference $M=\eta_2-\eta_1$. Combining this, \Cref{theo:xi_real_oder_xi_fullinfo}\eqref{eqn:theo:connect_sigma} and \Cref{theo:xi_real_oder_xi_fullinfo}\eqref{eqn:theo:connect_gamma} yields that $M$, $\sigma_1^2$ and $a=\nu_2(\R)$ serve to understand the distribution of $\xi_1$ and $\xi_2$ completely. We will see that these three are determined by the limits of the sums given by \eqref{eqn:I1} and \eqref{eqn:I2}. To give a first impression why this is the case we explain briefly the impact of $I_{n,1,x}$. Since the summands of $T_n$ fulfill the so-called condition of infinite smallness, i.e. a finite number of summands has no influence of the sum's convergence behaviour, well-known limit theorems to infinitely divisible distributed random variable can be applied, see, for instance, \citet{gndedenKolmogorov}. In the case of real-valued $\xi_1$ we obtain from these theorems
\begin{align}\label{eqn:sumP_conv_eta1}
\sum_{i=1}^{k_n}P_{n,i}\Bigl( \varepsilon_{n,i}\frac{\mathrm{ d } \mu_{n,i}}{\,\mathrm{ d }P_{n,i}}>e^x-1+\varepsilon_{n,i} \Bigr) \to \eta_1(x,\infty)
\end{align}
for all x from a dense subset of $(0,\infty)$. If additionally $\xi_2$ is real valued then the same holds for $\eta_2$ when we replace $P_{n,i}$ by $Q_{n,i}$. Combining these and \eqref{eqn:maxeps_to_0} shows that $I_{n,1,e^x-1}$ tends to $M(x,\infty)=(\eta_2-\eta_1)(x,\infty)$ for all x coming from a dense subset of $(0,\infty)$ if both, $\xi_1$ and $\xi_2$, are real-valued. In the case of $a=\nu_2(\R)=P(\xi_2\in\R)<1$ a similar convergence can be observed, namely $I_{n,1,e^x-1}$ tends to $(\eta_2-\eta_1)(x,\infty)+M(\infty)$, where the mass $M(\infty)$ in the point $\infty$ characterizes $a$ uniquely.

\begin{theorem}\label{theo:general_limit_theorem}
Let $I_{n,1,x}$ and $I_{n,2,x}$, $x>0$, be defined as in \eqref{eqn:I1} and \eqref{eqn:I2}. $\xi_1$ is real-valued if and only if the following \eqref{enu:theo:general_limit_theorem_levy_mass} and \eqref{enu:theo:general_limit_theorem_sigma} hold:
\begin{enumerate}[(a)]
	\item\label{enu:theo:general_limit_theorem_levy_mass} There is a dense subset ${\mathcal D }$ of $(0,\infty)$ and a 
	measure $M$ on $((0,\infty],{\mathcal B }(0,\infty])$ such that for all $x\in {\mathcal D }$
	$$\lim_{n\to\infty} I_{n,1,e^x-1} = M(x,\infty].$$
	
	\item \label{enu:theo:general_limit_theorem_sigma} For some $\sigma^2\in[0,\infty)$ we have 
	\begin{align*}
	\lim_{x\searrow 0} \,\underset{n\to\infty}{ \substack{ \limsup \\ \liminf } }\; I_{n,2,x} = \sigma^2,
	\end{align*}
	i.e. this equation holds for $\limsup_{n\to\infty}$ and $\liminf_{n\to\infty}$ simultaneously.
\end{enumerate}
If \eqref{enu:theo:general_limit_theorem_levy_mass} and \eqref{enu:theo:general_limit_theorem_sigma} hold then using the notation from \Cref{theo:xi_real_oder_xi_fullinfo}\eqref{enu:theo:xi_real_oder_xi_fullinfo_infinit_divi} we obtain $\nu_2(\R)=\exp(-M(\{\infty\})),$ $\sigma^2=\sigma_1^2=\sigma_2^2$ and $\eta_2-\eta_1=M_{|(0,\infty)}$.
\end{theorem}
\begin{remark}\label{rem:general_theor}
\begin{enumerate}[(i)]
	\item\label{enu:rem:general_theor_density} From \Cref{theo:xi_real_oder_xi_fullinfo}\eqref{eqn:theo:connect_Levymeasure} we get for all $x>0$
	\begin{align}\label{eqn:conn_eta_M}
	\frac{\mathrm{ d } \eta_1}{\,\mathrm{ d }M}(x)=\frac{1}{\exp(x) - 1}\textrm{ and }\frac{\mathrm{ d } \eta_2}{\,\mathrm{ d }M}(x)=\frac{\exp(x)}{\exp(x)-1}.
	\end{align}
	
	\item\label{enu:rem:general_theor_max} Consider the rowwise identical case with a noise distribution independent on $n$, i.e. $P_{n,i}=P_{0}$, $\mu_{n,i}=\mu_{n}$ and $\varepsilon_{n,i}=\varepsilon_{n}$. Thus, $Y_{n,1},\ldots,Y_{n,k_n}$ are identical $P_0$-distributed under the null.  By using techniques of extreme value theory it is sometimes possible to show that 
	\begin{align*}
	\max_{1\leq i \leq k_n}\Bigl\{ \varepsilon_{n}\frac{\mathrm{ d } \mu_{n}}{\,\mathrm{ d }P_{0}}(Y_{n,i})\Bigr\} \overset{\mathrm d}{\longrightarrow} \widetilde Y
	\end{align*}
	for a real-valued random variable $\widetilde Y$. Note that  $\max_{1\leq i \leq k_n}\{P_{n,i}(\varepsilon_{n,i}\frac{\mathrm{ d } \mu_{n,i}}{\,\mathrm{ d }P_{n,i}}>\tau)\}\leq \tau^{-1}\max_{1\leq i \leq k_n}\varepsilon_{n,i}\to 0$. Hence, regarding \eqref{eqn:sumP_conv_eta1} we get the following connection to the L\'{e}vy measure $\eta_1$ of $\xi_1$:
	\begin{align*}
	P(Y>e^x-1 ) = \exp(-\eta_1(x,\infty))
	\end{align*}
	for all $x$ coming from a dense subset of $(0,\infty)$.	This may be useful to get a first impression how to choose $\mu_{n}$ and $\varepsilon_{n}$ to obtain nontrivial limits.
\end{enumerate}
\end{remark}
%
%
\subsection{Asymptotic relative efficiency}\label{sec:ARE}
In the case of normal distributed limits we have
\begin{align}\label{eqn:Tn_normal}
T_n {\overset{\mathrm d}{\longrightarrow}}\left\{ \begin{array}{ll}
\xi_1\sim N(-\sigma^2/2,\sigma^2)\textrm{ under }P_{(n)}\text{ (null)},\\
\xi_2\sim N(\sigma^2/2,\sigma^2)\textrm{ under }Q_{(n)}\text{ (alternative)},
\end{array}\right. 
\end{align}
for some $\sigma\in[0,\infty)$, where $N(0,0)$ denotes the Dirac measure $\epsilon_0$ centered in $0$. In the case of $\sigma=0$ no test sequence can separate between the null and the alternative asymptotically, see \Cref{sec:trivial_limits}. Observe that both normal distributed limits depend only on one parameter, namely $\sigma^2$. In \Cref{appendix:additional}, see \Cref{theo:normal_limits}, we give many different equivalent conditions for normal distributed $\xi_1$ and $\xi_2$, even the conditions in \Cref{theo:xi_real_oder_xi_fullinfo} can be simplified in this case. Further equivalent conditions and closely related results can be found in Section A3 and A4 of \citet{Janssen1990}. 
In this section we restrict ourselves to these kind of limits, excluding the trivial case $\sigma=0$, and discuss the LLR test's power behaviour if the "wrong" signal distributions and/or the "wrong" signal probabilities  are chosen for the test statistic. To be more specific, we fix the triangular schemes of noise distributions $\{P_{n,i}:1\leq i \leq n\in\N\}$ and consider for $j=1,2$ a triangular scheme of signal distributions $\boldsymbol{\mu^{(j)}}=\{\mu_{n,i}^{(j)}:1\leq i \leq n\in\N\}$ as well as one of signal probabilities $\boldsymbol{\varepsilon^{(j)}}=\{\varepsilon_{n,i}^{(j)}:1\leq i \leq n\in\N\}$. Let $\boldsymbol{\theta_1}=(\boldsymbol{\mu^{(1)}},\boldsymbol{\varepsilon^{(1)}})$ be the true, underlying model and $\boldsymbol{\theta_2}=(\boldsymbol{\mu^{(2)}},\boldsymbol{\varepsilon^{(2)}})$ be the model pre-chosen by the statistician for the LLR test. Denote by $T_{n}(\boldsymbol{\theta_j})$ and $\varphi_{n}(\boldsymbol{\theta_j})=\mathbf{1}\{T_{n}(\boldsymbol{\theta}_j)>c_{n,j}\}$ the LLR statistic and the LLR test for the model $\boldsymbol{\theta_j}$, $j=1,2$. Using Pitman's asymptotic relative efficiency, see \citet{HajekSidakSen}, we quantify the loss in terms of the asymptotic power if $\varphi_{n}(\boldsymbol{\theta_2})$ instead of the optimal $\varphi_{n}( \boldsymbol{\theta_1})$ is used.

\begin{theorem}[LLR power under Gaussian limits]\label{theo:ARE}
Suppose that  $T_n(\boldsymbol{\theta_j})$, $j\in\{1,2\}$, converges to Gaussian limits, compare to \eqref{eqn:Tn_normal}, with $\sigma_j>0$. Moreover, assume that for $j,m\in\{1,2\}$ the limit
\begin{align}\label{eqn:theo:ARE:gamma_def} 
	\gamma(\boldsymbol{\theta_j},\boldsymbol{\theta_r})= \lim_{n\to\infty} \sum_{i=1}^{k_n} \varepsilon^{(j)}_{n,i} \varepsilon_{n,i}^{(r)} \text{Cov}_{P_{n,i}}\Bigl( \frac{\mathrm{ d } \mu_{n,i}^{(j)}}{\,\mathrm{ d }P_{n,i}},\frac{\mathrm{ d } \mu_{n,i}^{(r)}}{\,\mathrm{ d }P_{n,i}} \Bigr)
\end{align}
exists in $\R$. Suppose that $\gamma(\boldsymbol{\theta_j},\boldsymbol{\theta_j})=\sigma_j^2$. Let the critical values $c_{n,j}$ be chosen such that both tests $\varphi_n(\boldsymbol{\theta_1})$ and $\varphi_n(\boldsymbol{\theta_2})$ are asymptotically exact of a pre-chosen size $\alpha\in(0,1)$, i.e. $E_{\mathcal H_{0,n}}(\varphi_n(\boldsymbol{\theta_1}))\to \alpha$. Then the asymptotic power of the pre-chosen LLR test $\varphi_{n}(\boldsymbol{\theta_2})$ under the alternative ${\mathcal H }_{1,n}(\boldsymbol{\theta_1})$  of the true, underlying model $\boldsymbol{\theta_1}$ is given by
\begin{align*}
	 E_{{\mathcal H }_{1,n}(\boldsymbol{\theta_1})}(\varphi_{n,\boldsymbol{\theta_2}})&\to \Phi\Bigl( \frac{\gamma(\boldsymbol{\theta_1},\boldsymbol{\theta_2})}{\sqrt{\gamma(\boldsymbol{\theta_2},\boldsymbol{\theta_2})}}+u_{\alpha} \Bigr)\\
	&= \Phi\Bigl( \mathrm{sign}(\gamma(\boldsymbol{\theta_1},\boldsymbol{\theta_2}))\sqrt{\gamma(\boldsymbol{\theta_1},\boldsymbol{\theta_1})\mathrm{ARE}}+u_{\alpha} \Bigr),\\
	&\text{where }\mathrm{ARE}=\frac{\gamma(\boldsymbol{\theta_1},\boldsymbol{\theta_2})^2} {\gamma(\boldsymbol{\theta_1},\boldsymbol{\theta_1})\gamma(\boldsymbol{\theta_2},\boldsymbol{\theta_2})}\in[0,1]
\end{align*}
is Pitman's asymptotic relative efficiency, see \citet{HajekSidakSen}.
\end{theorem}
\begin{remark}\label{rem:theo:ARE}
The assumption  $\gamma(\boldsymbol{\theta_j},\boldsymbol{\theta_j})=\sigma_j^2$ is connected to the classical Lindeberg-condition. It is often but not always fulfilled if \eqref{eqn:Tn_normal} holds. For example, it is violated in the case $\beta=3/4$ and $r=\rho(\beta)$ for the heterogeneous normal mixture model, which is discussed in \Cref{sec:normal_mix}. The good news are that by a truncation argument we find for every model $\boldsymbol{\theta}=(\boldsymbol{\mu},\boldsymbol{\varepsilon})$, for which \eqref{eqn:Tn_normal} holds, another $\boldsymbol{\widetilde\theta}=(\boldsymbol{\widetilde\mu},\boldsymbol{\widetilde\varepsilon})$ such that the limit $\gamma(\boldsymbol{\widetilde\theta},\boldsymbol{\widetilde\theta})$ from \eqref{eqn:theo:ARE:gamma_def} exists and equals $\sigma^2$ from \eqref{eqn:Tn_normal}, and, moreover, the test's asymptotic behaviour is not effected by replacing $\boldsymbol{\theta}$ by $\boldsymbol{\widetilde\theta}$. The details are carried out in \Cref{appendix:additional}, see Lemma \ref{lem:LAN_var_not_to_sigma_REPLACE}.
\end{remark}
Note that \Cref{theo:ARE} gives the sharp upper bound of the asymptotic power for all tests of asymptotic size $\alpha\in(0,1)$ if \eqref{eqn:Tn_normal} holds for the underlying model. The asymptotic relative efficiency ARE is a good tool to quantify the loss of power if the wrong LLR test is used. If $\mathrm{ARE}=1$ there is no loss of power by using $\varphi_{n}(\boldsymbol{\theta_2})$ and if $\mathrm{ARE}=0$ the test  $\varphi_{n}(\boldsymbol{\theta_2})$ cannot distinguish between the null and the alternative asymptotically. Consider for a moment the rowwise identical case, i.e. $P_{n,i}=P_{n,1}$, $\mu^{(1)}_{n,i}=\mu_{n,1}^{(1)}$ etc. If $\mathrm{ARE}\in(0,1)$ then, heuristically, $(1-\mathrm{ARE})\cdot 100\%$ of the observations are wasted. To be more specific, it can be shown that $\varphi_n(\boldsymbol{\theta}_2)$ based on all $k_n$ observations $(Y_{n,1},\ldots,Y_{n,k_n})$ achieves the same power as the optimal test does when only $m=[(1-ARE)k_n]$ observations $(Y_{n,1},\ldots,Y_{n,m})$ are used, where $[x]$ is the integer part of $x\in\R$.

\subsection{Violation of \eqref{eqn:absolute_conti}}\label{sec:violation}
Here, we discuss how to handle a violation of \eqref{eqn:absolute_conti}. This issue was already discussed by \citet{Cai_Wu_2014}, see their Section III.C, in terms of the Hellinger distance to determine the detection boundary. Their idea can be used for our purpose to determine, more generally, the limits of $T_n$, even on the boundary. Instead of the original model it is sufficient to analyse a "closely related" model for which \eqref{eqn:absolute_conti} is fulfilled.\\
By Lebesgues' decomposition, see Lemma 1.1 of \citet{Strasser1985}, there exist a constant $\kappa_{n,i}\in[0,1]$, a $P_{n,i}$-null set $N_{n,i}$ as well as probability measures $\widetilde \mu_{n,i}$ and $\nu_{n,i}$ such that $\widetilde{\mu}_{n,i}\ll P_{n,i},$ $\nu_{n,i}(N_{n,i})=1$ and $\mu_{n,i}= (1-\kappa_{n,i})\widetilde{\mu}_{n,i} + \kappa_{n,i}\nu_{n,i}$. Now, let $\widetilde Q_{n,i}$, $\widetilde Q_{(n)}$ and $\widetilde T_n$ defined as $ Q_{n,i}$, $ Q_{(n)}$ and $ T_n$ replacing $\mu_{n,i}$ and $\varepsilon_{n,i}$ by $\widetilde\mu_{n,i}$ and $\widetilde\varepsilon_{n,i}$, respectively. Clearly, for this new model \eqref{eqn:absolute_conti} is fulfilled and our results can be applied to determine the limits of $\widetilde T_n$. When knowing these we can immediately give the ones of $T_n$: 
\begin{corollary}\label{cor:violation}
Suppose that \eqref{eqn:def_Tn} is fulfilled for $\widetilde T_n$, $\widetilde \xi_1$ and $\widetilde \xi_2$. Moreover, assume that $\sum_{i=1}^{k_n}\varepsilon_{n,i}\kappa_{n,i}\to c\in[0,\infty]$. Then \eqref{eqn:def_Tn} holds for $ T_n$, $\xi_1=\widetilde \xi_1-c$ and $\xi_2=\widetilde\xi_2 + X$, where $X$ is independent of $\widetilde \xi_2$ with $P(X=-c)=e^{-c}$ and $P(X=\infty)=1-e^{-c}$. In particular, $\xi_1\equiv -\infty$ and $\xi_2\equiv \infty$ if $c=\infty$, or if $\widetilde \xi_1\equiv -\infty$ and $\widetilde\xi_2\equiv \infty$. 
\end{corollary}
We can state the results of \Cref{cor:violation} also in terms of distributions. Denote by $\widetilde\nu_j$ the distribution of $\widetilde\xi_j$. Then $\nu_1=\widetilde\nu_1*\epsilon_{-c}$ and $\nu_2=e^{-c}\widetilde\nu_2*\epsilon_{-c}+(1-e^{-c})\epsilon_\infty$.

\section{Power of the higher criticism test}\label{sec:HCT}
In the previous section we discussed the LLR test which can be used to detect simple alternatives from the null. An adaptive and applicable test for alternatives of the whole completely detectable area is Tukey's HC test  modified by \citet{DonohoJin2004}. 
There are different versions of it. We prefer the one dealing with continuously distributed $p$-values $(p_{n,i})_{i\leq k_n}$ and consider $P_{n,i}=\lebesgue_{|(0,1)}$ having a quantile transformation $p_{n,i}=P_{n,i}((Y_{n,i},\infty))$ or $p_{n,i}=P_{n,i}((-\infty,Y_{n,i}])$ in mind. The optimality of HC in a discrete model, namely the Poisson means model, was shown by \citet{AriasWang2015}. Our results about the LLR statistic in \Cref{sec:LLRT} are valid for discrete models but in this section we only regard continuous ones. The extension to discrete models is a possible project for the future. \\
The HC statistic for outcomes $p_{n,i}\in[0,1]$ is defined by
\begin{align*}
HC_{n}=\sup_{t\in(0,1)}\Bigl | \sqrt{k_n}\;\frac{\mathbb{F}_n(t)-t}{\sqrt{t(1-t)}} \Bigr|,
\end{align*}
where $\mathbb{F}_n$ is the empirical distribution function of the observation vector $(p_{n,i})_{i\leq k_n}$. For every $t\in(0,1)$ we compare the empirical distribution function and the null/noise distribution function $t\mapsto F(t)=t$. This difference is normalized   in the spirit of the central limit theorem. For a fixed $t$ the resulting fraction is asymptotically standard normal distributed. The interval $(0,1)$, over which the supremum is taken, can be replaced by $(0,\alpha_0)$, $(k_n^{-1},\alpha_0)$ or $(k_n^{-1},1-k_n^{-1})$ for some tuning parameter $\alpha_0\in(0,1)$, see \citet{DonohoJin2004}. The test statistic can also be defined without taking the absolute value of the fraction. All these versions of the HC statistic would lead here to the same power results. To improve the readability of this section we give the results only for the HC version introduced above. By \citet{Jaeschke1979}, see also \citet{Eicker1979}, the limit distribution of $HC_n$ is known under the null. We have
\begin{align}\label{eqn:HC_limit_distr_null}
P_{(n)}( a_nHC_n-b_n \leq x ) \to \Lambda(x)^2=\exp(-2\exp(-x)),\;x\in\R,
\end{align}
where $\Lambda$ is the distribution function of a standard Gumbel distribution and the following normalisation constants are used 
\begin{align*}
a_n= \sqrt{ 2 \log\log(k_n) }\textrm{ and }b_n= 2 \log\log (k_n) + \frac{1}{2} \log\log\log (k_n ) - \frac{1}{2}\log (\pi).
\end{align*}
Hence, the test $\varphi_{n,HC,\alpha}=\mathbf{1}\{HC_n>c_n(\alpha)\}$ with
\begin{align*}
c_n(\alpha)=\frac{- \log ( -\log ( \alpha)/2 )+b_n}{a_n}=\sqrt{2\log\log (k_n)}(1+o(1)) 
\end{align*}
is an asymptotically exact level $\alpha\in(0,1)$ test, i.e. $E_{\mathcal H_{0,n}}(\varphi_{n,HC,\alpha})\to\alpha$. But we cannot recommend to use these critical values based on the limiting distribution since the convergence rate is really slow, see \citet{KhmaladzeShin2001}. Since the noise distribution is known, standard Monte-Carlo simulations can be used to estimate the $\alpha$-quantile of $HC_n$ for finite sample size. Alternatively, you can find finite recursion formulas for the exact finite distribution in the paper of \citet{KhmaladzeShin2001}.\\
In the following we present our tool for HC.
\begin{theorem}[Completely detectable by HC]\label{theo:HC_full_power}
Define for all $v\in(0,1/2)$ 
\begin{align}\label{eqn:HC_defi_Hn(v)}
H_n(v)= \frac{  | \sum_{i=1}^{k_n} \varepsilon_{n,i} (\mu_{n,i} (0,v] - v ) |   + 	 | \sum_{i=1}^{k_n} \varepsilon_{n,i} (\mu_{n,i} (1-v,1) - v ) | }{\sqrt{k_nv}}. 
\end{align}
Let $(v_n)_{n\in\N}$ be a sequence in the interval $(0,1/2)$ such that $a_n^{-1}H_n(v_n)\to\infty$ and $\liminf_{n\to\infty}k_nv_n>0$. Then $a_nHC_n-b_n \to \infty$ in $Q_{(n)}$-probability.
\end{theorem}
Basically, we compare the tails near to $0$ and $1$ of the signal and the noise distribution. This verification method for HC's optimality is an extension of the ones used by \cite{CaiJengJin2011,DonohoJin2004}.  Under the assumptions of \Cref{theo:HC_full_power} the sum of HC's error probabilities tends to $0$ for appropriate critical values. In other words, HC can completely separate the null and the alternative.\\
The same $H_n(v)$ can be used to show that HC has no power under the alternative, i.e. the sum of error probabilities tends to $1$ independently how the critical values are chosen. 
\begin{theorem}[Undetectable by HC]\label{theo:HC_undetect}
Suppose that $P_{n,i}=P_{n}$, $\varepsilon_{n,i}=\varepsilon_{n}$ and $\mu_{n,i}=\mu_{n}$ do not depend on $i$. Define $H_n(v)$ as in \Cref{theo:HC_full_power}. Moreover, assume that $P_{(n)}$ and $Q_{(n)}$ are mutually contiguous, compare to Remark \ref{rem:contiguous}. If
\begin{align}
&a_n\sup\{ H_n(v):v\in[r_n,s_n]\cup[t_n,u_n]\}\to 0,\textrm{ where }\label{eqn:HC_undect_condition}\\
&\frac{\log( r_n)}{\log(k_n)}\to -1, \;\frac{\log (u_n)}{\log(k_n)}\to 0,\; \textrm{ and }\;\frac{\log (s_n)}{\log(k_n)},\frac{\log (t_n)}{\log(k_n)}\to \kappa\in(0,1)\label{eqn:HC_cond_rn_un_sn_tn} 
\end{align}  
for some sequences $r_n,s_n,t_n,u_n\in(0,1)$ then
\begin{align}\label{eqn:HC_undetect}
Q_{(n)}(a_n HC_n-b_n \leq x ) \to \Lambda(x)^2=\exp(-2\exp(-x)),\;x\in\R.
\end{align} 
\end{theorem}
\begin{remark}\label{rem:HC}
Suppose that $a_n^2\sum_{i=1}^{k_n}\varepsilon_{n,i}^2\to 0$, which is usually fulfilled for sparse signals. From H{\"o}lder's inequality $(a_n /\sqrt{k_n})\sum_{i=1}^{k_n}\varepsilon_{n,i}\to 0$ follows. Hence, it is easy to see that the statements of \Cref{theo:HC_full_power,theo:HC_undetect} remain true if $H_n(v)$ is replaced by
\begin{align*}
\widetilde H_n(v)= \frac{1}{\sqrt{k_nv_n}}\Bigl( \sum_{i=1}^{k_n}\varepsilon_{n,i}(\mu_{n,i}(0,v]+\mu_{n,i}(1-v,1) ) \Bigr),\;v\in\Bigl(0,\frac{1}{2}\Bigr).
\end{align*}
\end{remark}

\section{Application to practical detection models} \label{sec:applications}

\subsection{Nonparametric alternatives for $p$-values}\label{sec:h-model}
Here, we discuss a generalisation of the $p$-values model \eqref{eqn:ill_density}. In particular, we suppose $P_{n,i}=\lebesgue_{|(0,1)}$. In contrast to \Cref{sec:intro_h}, we now consider that the shape function $h_{n,i}$, the shrinking parameter $\kappa_{n,i}>0$ and the signal probability $\varepsilon_{n,i}$ may depend on $i$. The assumption that the signal distribution has a shrinking support can be too restrictive for practice. But the approach allows an extension of the model in the way that we add a perturbation $r_{n,i}$. Throughout this section we consider signal distributions $\mu_{n,i}$ given by
\begin{align}\label{eqn:pertubation}
\frac{\mathrm{ d } \widetilde \mu_{n,i}}{\,\mathrm{ d }\lebesgue_{|(0,1)}}(u) = \frac{1}{\kappa_{n,i}}h_{n,i}\Bigl( \frac{u}{\kappa_{n,i}} \Bigr) + r_{n,i}(u)\geq 0 \textrm{ with }\int_0^1 r_{n,i} \,\mathrm{ d }\lebesgue=0,
\end{align}
where $h_{n,i}$ is close to some $h\in L^1(\lebesgue_{|(0,1)})$ and the perturbation $r_{n,i}$ is "small" in the sense that
\begin{align}\label{lem:h-model_rni_condition}
\sum_{i=1}^{k_n}\varepsilon_{n,i}^2 \int_0^1 r_{n,i}^2 \,\mathrm{ d }\lebesgue \to 0.
\end{align}
Instead of \eqref{eqn:intro_h_tau+eps} we suppose that 
\begin{align*}
\max_{1\leq i \leq k_n}( \varepsilon_{n,i}+\kappa_{n,i} ) \to 0.
\end{align*}
Since we already presented the results concerning this model for the rowwise identical case $\mu_{n,i}=\mu_n$ and $\varepsilon_{n,i}=\varepsilon_n$ in \Cref{sec:intro_h}, the theorems are stated only in their general versions here.
\begin{theorem}\label{theo:h-model}
Suppose that 
\begin{align}\label{eqn:h_unif_conv}
\sum_{i=1}^{k_n} \frac{\varepsilon_{n,i}^2}{\kappa_{n,i}}\to K\in[0,\infty]\textrm{ and }
\max_{1\leq i \leq k_n} \int_0^1 (h_{n,i}-h)^2 \,\mathrm{ d }\lebesgue \to 0
\end{align} 
for some $h,h_{n,i}\in L^2(\lebesgue_{|(0,1)})$. Without loss of generality we can suppose that 
\begin{align*}
\frac{\varepsilon_{n,1}}{\kappa_{n,1}}\leq \frac{\varepsilon_{n,2}}{\kappa_{n,2}}\leq \ldots\leq \frac{\varepsilon_{n,k_n}}{\kappa_{n,k_n}}.
\end{align*} 
\begin{enumerate}[(a)]
	\item\label{enu:theo:h-model_undetect} (Undetectable case) If $K=0$ then the undetectable case is present.
	
	\item\label{enu:theo:h-model_comp_detect} (Completely detectable case) If $K=\infty$, 
	\begin{align}\label{eqn:h-model_BED_rn}	
	\sum_{i=r_n}^{k_n} \varepsilon_{n,i} \to \infty \textrm{ and }\sum_{i=1}^{r_n} \frac{\varepsilon_{n,i}^2}{\kappa_{n,i}}\to\infty.
	\end{align}
	for some $r_n\in\{1,\ldots,k_n\}$ then we are in the completely detectable case.
	
	\item\label{enu:theo:h-model_nontrivial} If $\sup_{n\in\N}\sum_{i=1}^{k_n}\varepsilon_{n,i}<\infty$ or  $K<\infty$ then every accumulation point $\xi_1$ (in the sense of convergence in distribution) of $T_n$,  compare to \eqref{eqn:def_Tn}, is real-valued under the null. In particular, if $K\in(0,\infty)$  and 
	\begin{align}\label{eqn:h-model_max_condition}
	\max_{1\leq i \leq k_n} \frac{ \varepsilon_{n,i}}{ \kappa_{n,i} }=\frac{ \varepsilon_{n,k_n}}{ \kappa_{n,k_n}} \to 0
	\end{align}
	then the limits of $T_n$ are Gaussian and \eqref{eqn:Tn_normal} holds for $\sigma^2=\sigma^2(h)=K \int_0^1 h^2 \,\mathrm{ d }\lebesgue$. 
	
	\item\label{enu:theo:h-model_nonpara_LAN} In the spirit of \Cref{sec:ARE}, let $\boldsymbol{\theta_j}=\{( h_{n,i}^{(j)},\kappa_{n,i}^{(j)},\varepsilon_{n,i}^{(j)})_{i\leq k_n}:n\in\N\}$ denote a model for $j=1,2$ such that \eqref{eqn:h_unif_conv} and \eqref{eqn:h-model_max_condition} hold for some $ K^{(j)}\in(0,\infty)$ and $ h^{(j)}\in L^2(\lebesgue_{|(0,1)})$. Then all assumptions of \Cref{theo:ARE} are satisfied with
	\begin{align*}
	\gamma(\boldsymbol{\theta_1},\boldsymbol{\theta_2})=\lim_{n\to\infty}\sum_{i=1}^{k_n} \frac{ \varepsilon_{n,i}^{(1)} \varepsilon_{n,i}^{(2)}}{\kappa_{n,i}^{(1)}\kappa_{n,i}^{(2)}} \int_0^{\min\{\kappa_{n,i}^{(1)},\kappa_{n,i}^{(2)}\} } h_{n,i}^{(1)}( x/\kappa_{n,i}^{(1)} ) h_{n,i}^{(2)}( x/\kappa_{n,i}^{(2)} )\,\mathrm{ d }x
	\end{align*}
	if this limit exists.
\end{enumerate} 
\end{theorem}
Using \Cref{theo:h-model}\eqref{enu:theo:h-model_nonpara_LAN} we can calculate the asymptotic relative efficiency ARE if the LLR test $\varphi_n(\boldsymbol{\theta_2})$ is used although $\boldsymbol{\theta_1}$ is the underlying model. In the following we discuss two special cases in this context.
\begin{remark}\label{rem:chimeric_gamma}
Suppose the conditions of \Cref{theo:h-model}\eqref{enu:theo:h-model_nonpara_LAN} are fulfilled.
\begin{enumerate}[(i)]
	\item\label{enu:rem:chimeric_gamma_diff_eps} (No power under different shrinking) Assume that $\kappa_{n,i}^{(1)} (\kappa_{n,i}^{(2)})^{-1}$ converges uniformly for $i\in\{1,\ldots,k_n\}$ to $0$ or to $\infty$.
	From Cauchy Schwartz's inequality we get $\gamma(\boldsymbol{\theta_1},\boldsymbol{\theta_2})=0$ and, hence, $\mathrm{ARE}=0$.
	
	\item If $\varepsilon_{n,i}^{(1)}=\varepsilon_{n,i}^{(2)}$ and $\kappa_{n,i}^{(1)}=\kappa_{n,i}^{(2)}$ in \Cref{theo:h-model}\eqref{enu:theo:h-model_nonpara_LAN} then $\gamma(\boldsymbol{\theta_1},\boldsymbol{\theta_2})$ can be expressed in terms of $K^{(1)}=K^{(2)}$, $h^{(1)}$ and $h^{(2)}$. In particular, we obtain 
	\begin{align*}
	\mathrm{ARE}=  \frac{<h^{(1)}, h^{(2)}>^2}{<h^{(1)},h^{(1)}><h^{(2)}, h^{(2)}>}, \text{ where }<f,g>=\int_0^1 fg \,\mathrm{ d }\lebesgue.
	\end{align*}
\end{enumerate}
\end{remark}
If $\varepsilon_{n,i}=\varepsilon_{n}$ and $\kappa_{n,i}=\kappa_{n}$ does not depend on $i=1,\ldots,k_n$ then  \eqref{eqn:h-model_BED_rn} is fulfilled for $r_n=[k_n/2]$ if and only if $K=\infty$ and $k_n\varepsilon_{n}\to\infty$. Combining this and \Cref{theo:h-model} yields the detection boundary presented in \ref{enu:ill_exam_determ_detect_bound} from \Cref{sec:intro_h} and the Gaussian limits introduced in \ref{enu:intro_h_LLRT_on_db} on this boundary if $\beta<1$. Next, we give the generalisation of the result stated in \ref{enu:intro_h_nonreal_xi2} from \Cref{sec:intro_h} concerning the case $\beta=1$.
\begin{theorem}[Extreme case $\beta=1$]\label{theo:h-model_beta=1}
Let $\kappa_{n,i}=k_n^{-r}$, $r>0$, and   $\varepsilon_{n,i}=k_n^{-1}$. Let ${\mathcal D }$ be a dense subset of $(0,\infty)$ and $M$ be a measure on $(0,\infty]$ with $M(\{\infty\})=0$  such that $M(x,\infty)<\infty$ for all $x\in {\mathcal D }$ and 
\begin{align}\label{eqn:theo:h-model_beta=1_condi}
\max_{1\leq i \leq n}\Bigl| \int_0^1 h_{n,i}\mathbf{1}\{h_{n,i}>e^x-1\}\,\mathrm{ d }\lebesgue -M(x,\infty) \Bigr|\to 0.
\end{align}
Then \eqref{eqn:def_Tn} holds for $\xi_1$ and $\xi_2$ given as follows:
\begin{enumerate}[(a)]
	\item (Undetectable case) If $r<1$ then $\xi_1\equiv\xi_2\equiv 0$.
	
	\item\label{enu:theo:hmodel_beta=1_nonGauss} If $r=1$ then $\xi_j$, $j\in\{1,2\}$, is infinitely divisible with L\'{e}vy-Khintchine triplet $(\gamma_j,0,\eta_j)$, where $\gamma_j$ and $\eta_j$ are given by \eqref{eqn:gamma1}, \eqref{eqn:gamma_2_general} and \eqref{eqn:conn_eta_M}.
	
	\item If $r>1$ then $\xi_1\equiv -1$ and $\xi_2\sim e^{-1}\epsilon_{-1} \;\,+ (1-e^{-1})\epsilon_\infty$.
\end{enumerate}
\end{theorem}
\begin{remark}\label{rem:h-model_beta=1}
Let $h\in L^1(\lebesgue_{|(0,1)})$. Suppose that $h_{n,i}=h_n$, $\int_0^1 | h_{n}-h| \,\mathrm{ d }\lebesgue\to 0$ and $\lebesgue(u\in(0,1): h(u) = x)=0$ for all $x>0$. Note that the latter is always fulfilled for strictly monotone $h$.  Then \eqref{eqn:theo:h-model_beta=1_condi} holds for $M$ given by $M(x,\infty)=\int_0^1 h\, \mathbf{1}\{h>e^x-1\} \,\mathrm{ d }\lebesgue$.
Consequently, if $r=1$ then $\eta_1= \mathcal{L}(\log(h+1)|\lebesgue_{|(0,1)})$, or in other words $\xi_1 \overset{ }{=}\log(h(U)+1)$ for some uniformly distributed $U$ in $(0,1)$.
\end{remark} 
Note that we need for the statements in \Cref{theo:h-model_beta=1} only $h\in L^1(\lebesgue_{|(0,1)})$, and not $h \in L^2(\lebesgue_{|(0,1)})$ as in \Cref{theo:h-model}. It is also possible to determine the detection boundary if $h\notin L^2(\lebesgue_{|(0,1)})$. In this case we get nontrivial L\'{e}vy measures on the whole detection boundary depending heavily on the shape of $h$ comparable to the situation in \Cref{theo:h-model_beta=1}\eqref{enu:theo:hmodel_beta=1_nonGauss}. In the following we discuss an example for $h\in  L^1(\lebesgue_{|(0,1)})\setminus  L^2(\lebesgue_{|(0,1)})$.
\begin{theorem}\label{lem:h=xrho}
Let $h_{n,i}(x)=h(x)=(1-\alpha)x^{-\alpha}$ for all $x\in(0,1)$ and some $\alpha\in[1/2,1)$. Moreover, let $k_n=n$, $\varepsilon_{n,i}=n^{-\beta}$, $\beta\in(1/2,1)$, and $\kappa_{n,i}=n^{-r}$, $r>0$. Then the detection boundary is given by
\begin{align}\label{eqn:lem:h=xrho_detec_bound}
\rho^\#(\beta,\alpha)\; =\;\min \Bigl( 0, \frac{\beta-\alpha}{1-\alpha} \Bigr). 
\end{align} 
In detail, $r<\rho^\#(\beta,\alpha)$ (resp. $r>\rho^\#(\beta,\alpha)$) leads to the undetectable case (resp. completely detectable case). If $r=\rho^\#(\beta,\alpha)$ then $T_n$ converges to infinitely divisible $\xi_j$, $j\in\{1,2\}$, with L\'{e}vy-Khintchine triplet $(\gamma_j,0,\eta_j)$ under ${\mathcal H }_{0,n}$ and ${\mathcal H }_{1,n}$, respectively. $\gamma_j$ and $\eta_j$ are uniquely determined by \eqref{eqn:gamma1}, \eqref{eqn:gamma_2_general} and
\begin{align*}
\frac{\mathrm{ d } \eta_j }{\,\mathrm{ d }\lebesgue}(x)= \frac{(1-\alpha)^{\frac{1}{\alpha}}}{\alpha}e^x (e^x-1)^{-\frac{1}{\alpha}-1},\; x>0,
\end{align*}
\end{theorem}
Note that the limit in \Cref{lem:h=xrho} for $r=\rho^\#(\beta,\alpha)$ does not coincide with the one for $\beta=1$ from \Cref{theo:h-model_beta=1}\eqref{enu:theo:hmodel_beta=1_nonGauss} with $h_{n,i}(x)=(1-\alpha)x^{-\alpha}$. \\
Let us now consider the HC test. Since the given model is one for $p$-values the observations do not need to be transformed. Hence, the HC test is based on $p_{n,i}=Y_{n,i}$. 
\begin{theorem}[Higher criticism]\label{theo:HC_hmodel}
Consider the model 
\begin{enumerate}[(i)]
	\item\label{enu:theo:HC_hmodel_intro} from \Cref{sec:intro_h}, where $h\in L^{2+\delta}(\lebesgue_{|(0,1)})$ for some $\delta\in(0,1)$, or
	\item\label{enu:theo:HC_hmodel_h=xrho} from \Cref{lem:h=xrho}.
\end{enumerate} 
Then the areas of complete detection of the HC and the LLR test coincide. HC cannot distinguish between the null and the alternative asymptotically if $r\leq 1$ and $r=\rho(\beta)$ or $r=\rho^\#(\beta,\alpha)$, respectively, i.e. on the detection boundary. \\
Moreover, under the model assumptions of \Cref{theo:h-model_beta=1} with $h_{n,i}=h_n$ HC cannot distinguish between the null and the alternative asymptotically if $\beta=r=1$. 
\end{theorem}

\subsection{Heteroscedastic normal mixtures}\label{sec:normal_mix}
The heteroscedastic normal mixture model was already studied essentially in the literature, see \cite{CaiJengJin2011,DonohoJin2004,Ingster1997}. Nevertheless, we can give, as a further application of our results, some new insights about it concerning the extension of the detection boundary and the asymptotic power of the HC test on the boundary. But we first introduce the model. Let $k_n=n$, $P_{n,i}=P_0=N(0,1)$ and $\mu_{n,i}=\mu_n=N(\vartheta_n,\sigma_0^2)$, $\sigma_0>0$, where the parametrisation $\varepsilon_{n,i}=\varepsilon_{n}=n^{-\beta}$ and $\vartheta_n=\sqrt{2r\log n}$ with $\beta\in(1/2,1)$ and $r>0$ is used. The detection boundary given by
\begin{align}\label{eqn:detect_bound_normal}
&\rho(\beta,\sigma_0)= \left\{
\begin{array}{lll}
(2-\sigma_0^2)\left( \beta-\frac{1}{2} \right) &\textrm{if } \frac{1}{2}< \beta \leq 1- \frac{\sigma_0^2}{4},\,\sigma_0<\sqrt{2},&(\text{I}) \\
\left( 1-\sigma_0 \sqrt{1-\beta} \right)^2 &\textrm{if } 1-\frac{\sigma_0^2}{4} < \beta < 1,\,\sigma_0<\sqrt{2},&(\text{II})\\
0 &\textrm{if } \frac{1}{2}< \beta \leq 1-\frac{1}{\sigma_0^2},\,\sigma_0\geq \sqrt{2},& (\text{III})\\
\left( 1-\sigma_0 \sqrt{1-\beta} \right)^2 &\textrm{if } 1-\frac{1}{\sigma_0^2} < \beta < 1,\,\sigma_0\geq \sqrt{2},&(\text{IV})
\end{array}				\right.
\end{align}
and the limits of $T_n$ on it were already determined by \cite{CaiJengJin2011} and \cite{Ingster1997}. The detection boundary is plotted for different $\sigma_0$ in \Cref{fig:normal+expfam}. Moreover, it was shown that the completely detectable areas of the LLR and HC tests coincide, see \cite{CaiJengJin2011,DonohoJin2004}. All these results can be proven by using our methods, see \cite{Ditzhaus2017}. Note that the HC test is applied to the vector $(p_{n,i})_{i\leq k_n}$ of $p$-values, which we get by transforming each observations $Y_{n,i}$ to $p_{n,i}=1-\Phi(Y_{n,i})$. 
\begin{prop}[see Theorems 5 and 6 of \cite{CaiJengJin2011}]\label{theo:detectbou_norm_sparse}
\begin{enumerate}[(a)]	
	\item \label{eqn:norm_undetect} If $r<\rho(\beta,\sigma_0)$ then we are in the undetectable case, i.e. no test can distinguish between the null $\mathcal H_{0,n}$ and the alternative $\mathcal H_{n,1}$ asymptotically.
	
	\item \label{eqn:norm_comp_detect} The LLR as well as the HC test can completely separate the null and the alternative asymptotically. 
	
	\item Suppose that $r=\rho(\beta, \sigma_0)$. Moreover, add a logarithmic term in the parametrisation of $\varepsilon_n$ as follows:
	\begin{align}\label{eqn:eps_normal_quadra}
	\varepsilon_{n}=n^{-\beta}\left( \log (n) \right)^{E(\beta,\sigma_0)}\,\text{ with }E(\beta,\sigma_0)= 			\begin{cases}
	0 & \text{on (I)}. \\
	\frac{1}{2}-\frac{\sqrt{1-\beta}}{2\sigma_0} & \text{ else.}
	\end{cases}   
	\end{align}
	In the following we discuss the different parts (I), (II) and (IV) of the detection boundary.
	\begin{enumerate}[(i)]
		\item	\label{eqn:norm_hetero_LLR_linear} (Gaussian limits) Consider part (I). Then \eqref{eqn:Tn_normal} holds for 
		\begin{align*}
		\sigma^2= \Bigl( \sigma_0\sqrt{2-\sigma_0^2} \Bigr)^{-1} \Bigl( 1-\frac{1}{2}\mathbf{1}\Bigl\{\beta= 1-\frac{\sigma_0^2}{4}\Bigr\} \Bigr)
		\end{align*}
		
		\item\label{eqn:norm_hetero_LLR_quad} Consider the parts (II) and (IV). Then \eqref{eqn:def_Tn} holds for infinitely
		divisible $\xi_1$ and $\xi_2$ with L\'{e}vy-Khintchine triplets $(\gamma_1,0,\eta_1)$ and $(\gamma_2,0,\eta_2)$, respectively, where $\eta_1$, $\eta_2$ are given by
		\begin{align*}
		&\frac{\mathrm{ d }\eta_1}{\mathrm{ d }\lebesgue}(x)=\frac{1}{c_1 } \left( e^x-1 \right)^{c_2-3} \,e^x\textrm{ and }\frac{\mathrm{ d }\eta_2}{\mathrm{ d }\lebesgue}(x)= e^x\,\frac{\mathrm{ d }\eta_1}{\mathrm{ d }\lebesgue}(x),\,x>0,
		\end{align*}
		with $c_1=2\sqrt{\pi}\sigma_0^{c_3}c_4$, $c_2=c_4^{-1}(\sigma_0-2\sqrt{1-\beta})$, $c_3=c_4^{-1}\sigma_0-\sqrt{1-\beta}$ and $c_4= \sigma_0-\sqrt{1-\beta}$,  
		and $\gamma_1$ and $\gamma_2$ fulfill \eqref{eqn:gamma1} and \eqref{eqn:gamma_2_general} with $\sigma^2=0$.
	\end{enumerate}
\end{enumerate}
\end{prop}
\begin{remark}
By carefully reading the proof of \cite{CaiJengJin2011}, see in particular the top of page 658, there must be an additional factor $1/2$ in the exponent of the logarithmic term in their definition of $\varepsilon_{n}$ as in our \eqref{eqn:eps_normal_quadra}. 
\end{remark}
Applying our \Cref{theo:HC_undetect} we can show, as already postulated, that HC has no asymptotic power on the boundary.
\begin{theorem}[HC on the boundary]\label{theo:heterosce_normal_HCT}
Let $r=\rho(\beta,\sigma_0)>0$, $\beta\in (1/2,1)$. Moreover, reparametrize $\varepsilon_{n}$ on the quadratic part of the boundary as we did in \eqref{eqn:eps_normal_quadra}. Then the HC test has no (asymptotic) power, whereas the LLR does so. 
\end{theorem} 
In \eqref{eqn:detect_bound_normal} the detection boundary is (only) defined for $\beta<1$. As we already did in the previous section, we can extend this boundary for $\beta=1$ by a infinite vertical line starting in $(r,\beta)=(1,1)$, see \Cref{fig:normal+expfam}. Again, we observe on this line unusual limits of $T_n$.
\begin{theorem}[Detection boundary extension]\label{theo:normal_ext_boundary}	
\begin{enumerate}[(i)]
	\item The pair $(\beta,r)$ with $\beta=1$ and $r<1$ belongs to the undetectable region.
	
	\item If $\beta=1$ and $r=1$ then $\xi_1\equiv -1/2$ and $\xi_2\sim e^{-1/2}\epsilon_{-1/2} + (1-e^{-1/2})\epsilon_\infty$.
	
	\item If $\beta=1$ and $r>1$ then $\xi_1\equiv -1$ and $\xi_2\sim e^{-1}\epsilon_{-1} + (1-e^{-1})\epsilon_\infty$.
\end{enumerate}
\end{theorem}
\begin{figure}[tb] 
\begin{center}	
	\includegraphics[trim = 22mm 150mm 80mm 10mm, clip, width=0.45\textwidth]{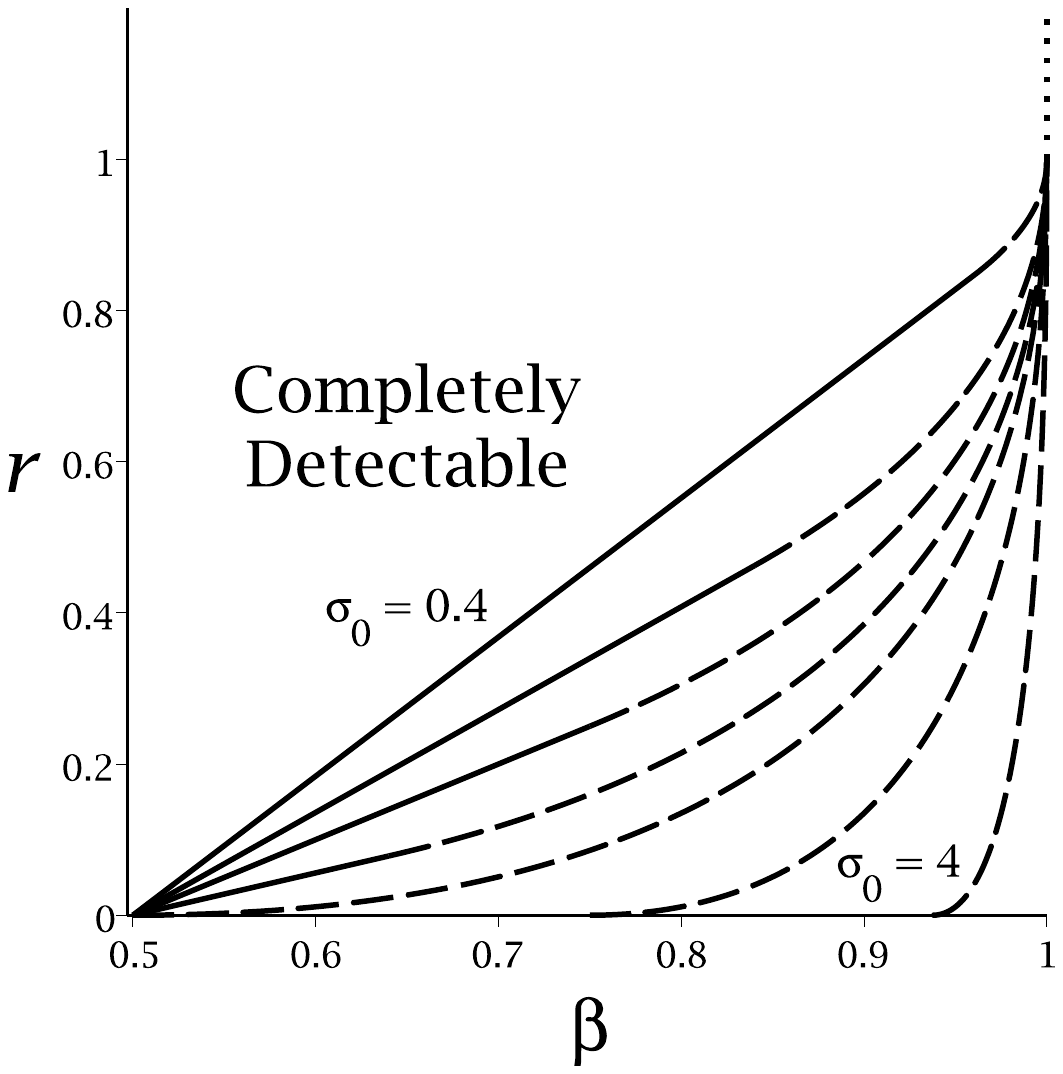}\quad
	\includegraphics[trim = 22mm 150mm 80mm 10mm, clip, width=0.45\textwidth]{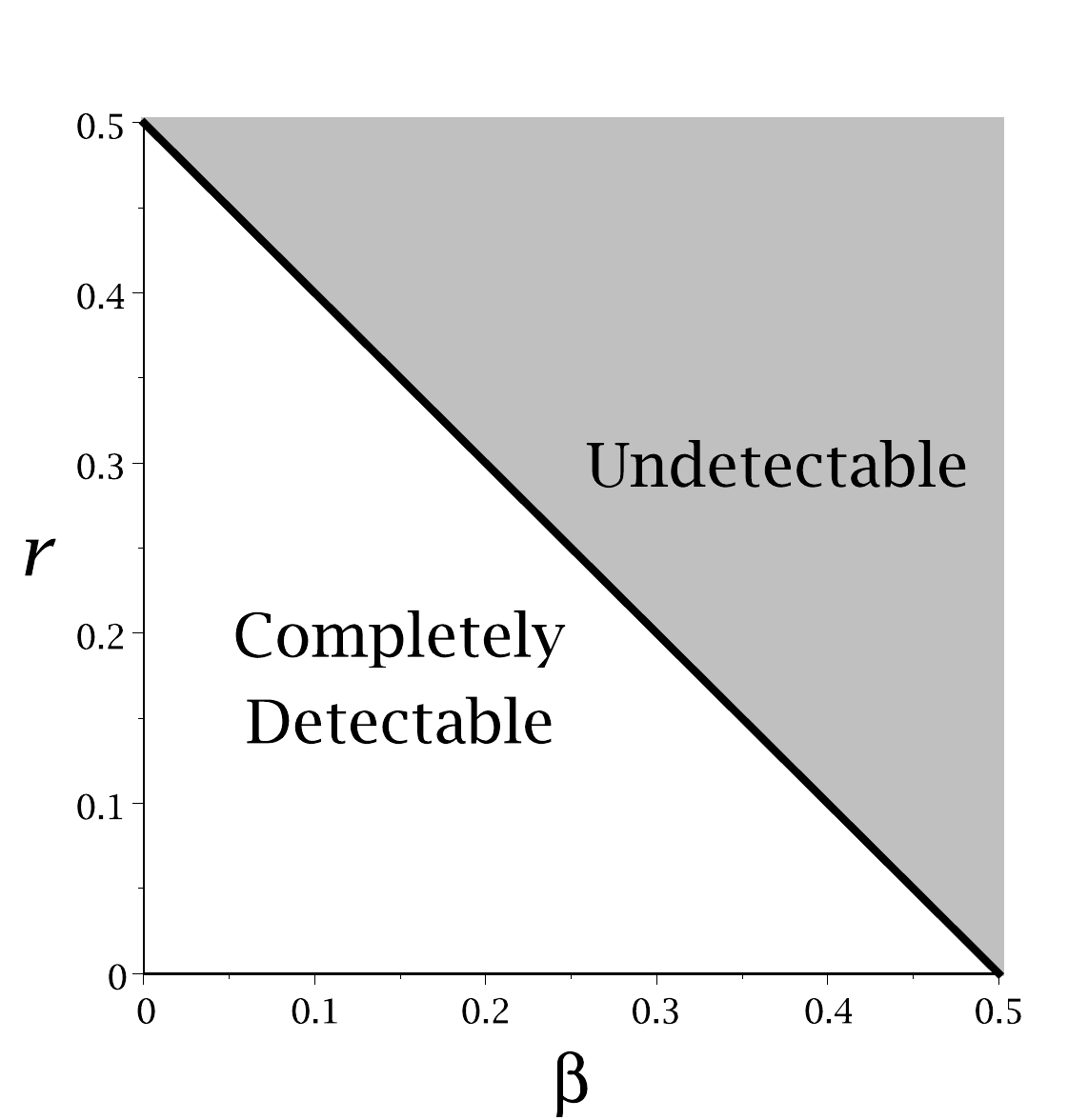}
\end{center}
\caption[Detection boundaries of the sparse and dense heteroscedastic normal mixture model]{Detection boundaries for the heteroscedastic normal mixture model. 
	Left: (Sparse case for $\sigma_0\in\{0.4,0.8,1,1.2,\sqrt{2},2,4\}$) Above the boundary is the completely detectable area and underneath is the undetectable area for both tests (LLR and HC). The limits $\xi_1$ and $\xi_2$ are Gaussian on the linear part (solid) and non-Gaussian on the quadratic part (dashed). In both cases the HC test has no asymptotic power. On the vertical dotted line $P(\xi_2\in\R)\in(0,1)$. Right: (dense case for $\sigma_0^2=1$) Above the boundary is the undetectable area and underneath is the completely detectable area for both tests. On the boundary the limits $\xi_1$ and $\xi_2$ are Gaussian and the HC test has no power.} \label{fig:normal+expfam}
\end{figure}
The results concerning ARE can also be applied for the heteroscedastic models. Fix the variance parameter $\sigma_0>0$. Let $\boldsymbol{\theta_1}=(\beta_1,r_1)$ and $\boldsymbol{\theta_2}=(\beta_2, r_2)$ represent two models from the linear part (I) of the detection boundary leading to Gaussian limits of $T_n$. Suppose that the models are different, i.e. $\beta_1\neq \beta_2$.  By applying \Cref{theo:ARE} and simple calculations, which are omitted to the reader, $\mathrm{ARE}=0$ can be shown. That means that the LLR test $\varphi_{n}(\boldsymbol{\theta_2})$ can not distinguish between the null and the alternative asymptotically when $\boldsymbol{\theta_1}$ is the true, underlying model. As already mentioned $\gamma(\boldsymbol{\theta_j},\boldsymbol{\theta_j})=\sigma^2_j$ does not hold  if $\beta_j= 1- \sigma_0^2/4$. In this case make use of the truncation Lemma \ref{lem:LAN_var_not_to_sigma_REPLACE}.\\
\citet{CaiJengJin2011} already considered the dense case $\beta<1/2$. In this case $\sigma_0^2\neq 1$ always leads to the completely detectable case independently of how the signal strength $\vartheta_n$ is chosen. Thus, only the heterogeneous case $\sigma_0^2=1$ is of real interest. In this case the parametrisation $\vartheta_n=n^{r}$ is used for $r>0$. The corresponding detection boundary is given by $\rho(\beta)=1/2-\beta$ and is plotted in \Cref{fig:normal+expfam}. The HC test achieves the same region of complete detection, see \cite{CaiJengJin2011}. Our results concerning the tests' power behaviour on the detection boundary can also be applied. In short, on the detection boundary \eqref{eqn:Tn_normal} holds for some $\sigma>0$ and the HC test has no asymptotic power there. This is even possible to a general class of one-parametric exponential families including the dense heterogeneous normal mixtures. To not overload this paper, we omit further details concerning the dense case and refer the reader to the thesis of \citet{Ditzhaus2017}. 

\appendix
\section{Gaussian limits}\label{appendix:additional}

Gaussian limits $\xi_1$ and $\xi_2$, compare to \eqref{eqn:Tn_normal}, are of special interest, for example regarding \Cref{theo:ARE}. Recall that the degenerate case is included as $\sigma=0$. In the following we give several equivalent conditions for Gaussian limits.
\begin{theorem}[Gaussian limits]\label{theo:normal_limits}
The conditions \eqref{enu:theo:normal_limits_both_gauss}-\eqref{enu:theo_normal_limits_my_cond} are equivalent:
\begin{enumerate}[(a)]
	\item\label{enu:theo:normal_limits_both_gauss} $\xi_1$ and $\xi_2$ are Gaussian or $\xi_1=\xi_2\equiv 0$ with probability one.
	
	\item\label{enu:theo:normal_limits_xi1} $\xi_1\sim N( -\frac{\sigma^2}{2},\sigma^2 )$ for some $\sigma^2\in[0,\infty)$.
	
	\item\label{enu:theo:normal_limits_xi2} $\xi_2\sim N( \frac{\sigma^2}{2},\sigma^2 )$ for some $\sigma^2\in[0,\infty)$.
	
	\item \label{enu:theo:normal_limits_xi2_real_xi1_normal} $\xi_2$ is real-valued and $\xi_1\sim N(a,\sigma^2)$ for some $a\in\R$, $\sigma^2\in[0,\infty)$.
	
	\item\label{enu:theo:normal_limits_xi2_normal} $\xi_2\sim N(a,\sigma^2)$ for some $a\in\R$, $\sigma^2\in[0,\infty)$.
	
	\item\label{enu:theo:normal_limits_Zn} $Z_n$ given by \eqref{eqn:def_Zn} converges in distribution under $P_{(n)}$ to some normal distributed $Z\sim N(0,\sigma^2)$ for some $\sigma^2\in[0,\infty)$:
	\begin{align}\label{eqn:def_Zn}
	Z_n= \sum_{i=1}^{k_n} \varepsilon_{n,i} \Bigl( \frac{\mathrm{ d } \mu_{n,i}}{\,\mathrm{ d }P_{n,i}} - 1\Bigr) \overset{\mathrm d}{\longrightarrow}Z.
	\end{align}
	
	\item\label{enu:theo:normal_limits_xi2_maxq} $\xi_2$ is real-valued and 
	$\max_{1\leq i \leq k_n}  \frac{\mathrm{ d } Q_{n,i}}{\,\mathrm{ d }P_{n,i}}  \to 1 \textrm{ in }P_{(n)}\textrm{-probability}$.
	
	\item\label{enu:theo:normal_limits_xi2_maxmu} $\xi_2$ is real-valued and $\max_{1\leq i \leq k_n} \varepsilon_{n,i}  \frac{\mathrm{ d } \mu_{n,i}}{\,\mathrm{ d }P_{n,i}}  \to 0 \textrm{ in }P_{(n)}\textrm{-probability}$.
	
	\item \label{enu:theo_normal_limits_my_cond} For some $\tau\in(0,\infty)$ and all $x>0$ we have $I_{n,1,x}\to 0$ and $I_{n,2,\tau}\to\sigma^2\in[0,\infty)$.
\end{enumerate}
If one of the conditions \eqref{enu:theo:normal_limits_xi1}--\eqref{enu:theo:normal_limits_Zn} or \eqref{enu:theo_normal_limits_my_cond} is fulfilled for some $\sigma^2\in[0,\infty)$ then the others do so for the same $\sigma^2$.
\end{theorem}
\begin{remark}\label{rem:LAN_distances}
\Cref{theo:normal_limits}\eqref{enu:theo_normal_limits_my_cond} holds for some $\tau>0$ if and only if it does for all.
\end{remark}
To apply \Cref{theo:ARE} $\gamma(\boldsymbol{\theta},\boldsymbol{\theta})=\sigma^2$ is needed, where $\sigma^2$ comes from the previous section and $\boldsymbol{\theta}$ denotes the underlying model, compare to the notation in \Cref{sec:ARE}. As already mentioned there are examples, for which this equation fails although $\xi_1$ and $\xi_2$ are normal distributed. But by truncation we can always ensure the equality without changing the asymptotic results.
\begin{lemma}[Truncation]\label{lem:LAN_var_not_to_sigma_REPLACE}
Let the assumptions of \Cref{theo:normal_limits} and one of its equivalent conditions \eqref{enu:theo:normal_limits_both_gauss}-\eqref{enu:theo_normal_limits_my_cond} be fulfilled. In order to use a truncation argument define 
\begin{align*}
\widetilde\varepsilon_{n,i}= \varepsilon_{n,i}\mu_{n,i}\Bigl( \varepsilon_{n,i}\frac{\mathrm{ d } \mu_{n,i}}{\,\mathrm{ d }P_{n,i}}\leq \tau \Bigr) \text{ for some }\tau>0
\end{align*}
and  let $\widetilde{\mu}_{n,i}$ be given as follows: if $\widetilde \varepsilon_{n,i}=0$ then $\frac{\mathrm{ d } \widetilde \mu_{n,i}}{\,\mathrm{ d }P_{n,i}}=1$, and otherwise 
\begin{align*}
\frac{\mathrm{ d } \widetilde\mu_{n,i}}{\,\mathrm{ d }P_{n,i}}=\frac{\mathrm{ d } \mu_{n,i}}{\,\mathrm{ d }P_{n,i}}\mathbf{1}\Bigl\{ \varepsilon_{n,i}\frac{\mathrm{ d } \mu_{n,i}}{\,\mathrm{ d }P_{n,i}} \leq\tau \Bigr\}\Bigl[ \mu_{n,i}\Bigl(\varepsilon_{n,i} \frac{\mathrm{ d } \mu_{n,i}}{\,\mathrm{ d }P_{n,i}}\leq \tau \Bigr) \Bigr]^{-1}.
\end{align*}
All our asymptotic results in this paper remain the same if we replace $\mu_{n,i}$ and $\varepsilon_{n,i}$ by $\widetilde \mu_{n,i}$ and $\widetilde \varepsilon_{n,i}$.  
\end{lemma}

\section{Proofs}\label{appendix:proofs}
	In the following we give all the proofs. These are not given in the order of their appearance since we apply, for example, \Cref{theo:general_limit_theorem} to verify \Cref{theo:xi_real_oder_xi_fullinfo}. Before giving the proofs we introduce some useful properties of binary experiments and generalise limit theorems of \citet{gndedenKolmogorov} to infinitely divisible distributions. 

\subsection{Binary experiments and distances for probability measures}\label{sec:binary_exp_+_dist}
Binary experiments classify different types of signal detectability. This gives us a first rough insight in the different detection regions for our signal detection problem. This standard approach is recalled for a sequence of binary experiments $\{ \widetilde P_{(n)} , {\widetilde Q }_{(n)}\}$, $n\in\N\cup\{0\}$, where the underlying measurable spaces $(\Omega_n,{\mathcal A }_n)$ may change with $n$. Recall the equivalence of the weak convergences in \eqref{eqn:bin_conv_def_nu1} and \eqref{eqn:bin_conv_def_nu2} on $[-\infty,\infty]$:
\begin{alignat}{2}
&{\mathcal L }\Bigl( \log \frac{\mathrm{ d } \widetilde Q_{(n)}}{\,\mathrm{ d }\widetilde P_{(n)}} \Bigr | \widetilde P_{(n)} \Bigr) \;&&\overset{\mathrm w}{\longrightarrow}\;{\mathcal L }\Bigl( \log \frac{\mathrm{ d } \widetilde Q_{(0)}}{\,\mathrm{ d }\widetilde P_{(0)}} \Bigr |\widetilde P_{(0)}  \Bigr)=\nu_1 \textrm{ (say)},\label{eqn:bin_conv_def_nu1}\\
&{\mathcal L }\Bigl( \log \frac{\mathrm{ d } \widetilde Q_{(n)}}{\,\mathrm{ d }\widetilde P_{(n)}} \Bigr | \widetilde Q_{(n)} \Bigr) \;&&\overset{\mathrm w}{\longrightarrow}\;{\mathcal L }\Bigl( \log \frac{\mathrm{ d } \widetilde Q_{(0)}}{\,\mathrm{ d }\widetilde P_{(0)}} \Bigr | \widetilde Q_{(0)} \Bigr)=\nu_2 \textrm{ (say)}.\label{eqn:bin_conv_def_nu2}
\end{alignat} 
Following Le Cam we say that $\{ \widetilde P_{(n)} , \widetilde{ Q }_{(n)}\}$ converges weakly to $\{\nu_1,\nu_2\}$  ($\{ \widetilde P_{(0)} , \widetilde{ Q }_{(0)}\}$, respectively) if and only if \eqref{eqn:bin_conv_def_nu1} or \eqref{eqn:bin_conv_def_nu2} is fulfilled. Note that every sequence of binary experiments has at least one accumulation point in the sense of weak convergence, see Lemma 60.6 of \citet{Strasser1985}. In general $\nu_1$ is a measure on $\R\cup\{-\infty\}$ and $\nu_2$ is one on $\R\cup\{\infty\}$ connected by
\begin{align}\label{eqn:relation_nu_2_nu_1}
\frac{\,\mathrm{ d }\nu_{2|\R}}{\,\mathrm{ d }\nu_{1|\R}}(x)=e^x \textrm{ and }\nu_2(\{-\infty\})= 1- \int e^x \,\mathrm{ d }\nu_1(x).
\end{align}
Using the terminology of weak convergence of binary experiments we can express the different types of (asymptotic) detectability as follows:
\begin{itemize}
	\item \textit{completely detectable:} $\{P_{(n)},Q_{(n)}\}$ converges weakly to the so called full informative experiment $\{\nu_1,\nu_2\}=\{\epsilon_{-\infty},\epsilon_{\infty}\}$.
	
	\item \textit{undetectable:} $\{P_{(n)},Q_{(n)}\}$ converges weakly to the so called uninformative experiment $\{\nu_1,\nu_2\}=\{\epsilon_0,\epsilon_0\}$.
	
	\item \textit{detectable:} None (weak) accumulation point of $\{P_{(n)},Q_{(n)}\}$ is the uninformative experiment $\{\nu_1,\nu_2\}=\{\epsilon_0,\epsilon_0\}$.
\end{itemize} 
The variational distance of probability measures $\widetilde P$ and $\widetilde Q$ on a common measure space $(\widetilde\Omega,\widetilde {\mathcal A })$ is given by
\begin{align}\label{eqn:defi_tot_var}
|| \widetilde {P }-\widetilde Q || = \sup\{  E_{\widetilde P }(\varphi) - E_{ \widetilde Q }(\varphi) : \textrm{ measurable }\varphi:\widetilde \Omega\to [0,1]\},
\end{align}
see Lemma 2.3 of \citet{Strasser1985}. It is easy to show that weak convergence of $\{\widetilde P_{(n)},\widetilde Q_{(n)}\}$ to $\{\widetilde P_{(0)},\widetilde Q_{(0)}\}$ implies convergence of the variational distance $||\widetilde P_{(n)}-\widetilde Q_{(n)}||\to ||\widetilde P_{(0)}-\widetilde Q_{(0)}||$. Our three cases can be reformulated to:
\begin{itemize}
	\item \textit{completely detectable:} $||P_{(n)}-Q_{(n)}||$ tends to $1$.
	
	\item \textit{undetectable:} $||P_{(n)}-Q_{(n)}||$ tends to $0$.
	
	\item \textit{detectable:} We have $\liminf_{n\to\infty}||P_{(n)}-Q_{(n)}||>0$.
\end{itemize} 
For product measures the Hellinger distance $d$ is useful:
\begin{align}\label{eqn:def_hellinger}
d^2(\widetilde P,\widetilde Q)=  \frac{1}{2} \int \Bigl( \Bigl(\frac{\,\mathrm{ d }\widetilde  P}{\,\mathrm{ d } \nu}\Bigr)^{\frac{1}{2}} - \Bigl({\frac{\,\mathrm{ d }\widetilde  Q}{\,\mathrm{ d } \nu} }\Bigr)^{\frac{1}{2}} \Bigr)^2 \,\mathrm{ d } \nu= 1-  \int  \Bigl(\frac{\,\mathrm{ d } \widetilde P}{\,\mathrm{ d } \nu}\frac{\,\mathrm{ d }\widetilde  Q}{\,\mathrm{ d } \nu} \Bigr)^{\frac{1}{2}} \,\mathrm{ d } \nu,
\end{align}
where $\widetilde P,\widetilde Q\ll \nu$. Since $d^2({\widetilde  P},{ \widetilde Q})\leq ||{\widetilde  P}-{\widetilde  Q}|| \leq \sqrt{2}\, d({\widetilde  P},{\widetilde  Q})$, see  Lemma 2.15 of \cite{Strasser1985}, we obtain from \eqref{eqn:def_qni} and \eqref{eqn:maxeps_to_0} that
\begin{align}\label{eqn:max_d2_to_0}
\max_{i=1,\ldots,k_n} d^2(P_{n,i},Q_{n,i}) \leq \max_{1\leq i \leq k_n} ||P_{n,i}-Q_{n,i}|| \leq \max_{1\leq i \leq k_n}\varepsilon_{n,i} \to 0.
\end{align}
Consequently, $d^2(P_{(n)}, Q_{(n)} ) = 1 - \prod_{i=1}^{k_n}(1- d^2(P_{n,i},Q_{n,i}))$ tends to $b\in[0,1]$ if and only if $-\log(1-b)$ is the limit of
\begin{align}\label{eqn:DN}
D_n=\sum_{i=1}^{k_n} d^2(P_{n,i},Q_{n,i}).
\end{align}
To sum up, we get the following characterisation of the trivial detection regions.
\begin{lemma}\label{lem:D_detection}
	\begin{enumerate}[(a)]
		
		\item We are in the undetectable case if and only if $D_n\to 0$. 
		
		\item\label{enu:lem:D_detection_comp} We are in completely detectable case if and only if $D_n\to \infty$.
		
	\end{enumerate}
\end{lemma}
Note that from the connection between the variational distance and the Hellinger distance we obtain 
\begin{align}\label{eqn:totvar_helg_absch_P+mu}
\frac{1}{2}\sum_{i=1}^{k_n}\varepsilon_{n,i}^2\:|| P_{n,i}- \mu_{n,i}||^2 \leq D_n\leq  \sum_{i=1}^{k_n} \varepsilon_{n,i}\:|| P_{n,i}- \mu_{n,i}||.
\end{align} 

\subsection{Limit theorems}
For the readers' convenience let us recall well known convergence results of \citet{gndedenKolmogorov} which we use rapidly. Let  $(Y_{n,i})_{1\leq i \leq k_n}$ be a triangular array of row-wise independent, infinitesimal, real-valued random variables on some probability space $(\Omega,{\mathcal A}, P)$. In our case we have 
\begin{align}\label{eqn:petrov_6_simp_cond_emptyset}
\sum_{i=1}^{k_n} P( Y_{n,i} \leq x) = 0
\end{align}
for all fixed $x<0$ if $n\geq N_x$ is sufficiently large. Combining this with (9) of Chap. 3.18, Theorem 4.25.4 and the subsequent remark of \cite{gndedenKolmogorov} yields:
\begin{theorem}\label{theo:petrov_6_simplify}
	We have distributional convergence 
	\begin{align*}
	\sum_{i=1}^{k_n} Y_{k_n,i}\overset{\mathrm d}{\longrightarrow}  Y 
	\end{align*}
	to some real-valued $Y$ on $\left(\Omega, \mathcal{A}, P\right)$ if and only if the following conditions \eqref{enu:petrov_6_simp_levy_mass}-\eqref{enu:petrov_6_simp_levy_gamma} hold. 
	\begin{enumerate}[(i)]
		\item\label{enu:petrov_6_simp_levy_mass} There is a L\'{e}vy measure $\eta$ on $\R\setminus\{0\}$ such that $\eta (-\infty,0) = 0$ and
		\begin{align*}
		\sum_{i=1}^{k_n} P\left(  Y_{k_n,i} > x \right)  \to  \eta (x,\infty)\in\R \textrm{ as }n\to\infty 
		\end{align*}
		for all $x \in C_+(\eta)$, i.e. for all continuity points of $t\mapsto \eta(t,\infty)$, $t>0$.
		
		\item\label{enu:petrov_6_simp_levy_sigma} There exists some constant $\sigma^2\in[0,\infty)$ such that
		\begin{align*}
		\sigma^2=\lim_{\varepsilon\searrow 0}\; \underset{n\to\infty}{ \substack{ \limsup \\ \liminf } }\; \sum_{i=1}^{k_n} \int_{\left\{\left| Y_{k_n,i} \right|<\varepsilon\right\}} Y_{k_n,i}^2\,\mathrm{ d }  P-\sum_{i=1}^{k_n} \left( \int_{\left\{\left| Y_{k_n,i} \right|<\varepsilon\right\}} Y_{k_n,i}  \,\mathrm{ d } { P} \right)^2.
		\end{align*}  
		
		\item \label{enu:petrov_6_simp_levy_gamma} 	There is some constant $\gamma\in\R$ and $\tau_0\in C_+(\eta)$ such that
		\begin{align*}
		&\lim_{n\to\infty} \sum_{i=1}^{k_n} \int  Y_{k_n,i}\mathbf{1}{\left\{\left| Y_{k_n,i} \right|<\tau_0\right\}}\,\mathrm{ d } P \\ &=\gamma + \int_{(-\tau_0,\tau_0)\setminus\{0\}} \frac{x^3}{1+x^2}\,\,\mathrm{ d } \eta(x)  - \int_{\R\setminus[-\tau_0,\tau_0]} \frac{x}{1+x^2}\,\,\mathrm{ d } \eta(x). \nonumber  
		\end{align*}
	\end{enumerate}
	Under \eqref{enu:petrov_6_simp_levy_mass}-\eqref{enu:petrov_6_simp_levy_gamma}  $Y$ is infinitely divisible with L\'{e}vy-Khintchine triplet $(\gamma,\sigma^2,\eta)$.
\end{theorem}
As stated in \Cref{theo:general_limit_theorem}, we have to deal also with positive weights in $\infty$ for the limits since $\nu_2=\rho+(1-a)\epsilon_{-\infty}$, where $a<1$ may occur.
\begin{theorem}\label{theo:petrov_including_infinity}
	Suppose that the conditions \eqref{enu:petrov_6_simp_levy_sigma} and \eqref{enu:petrov_6_simp_levy_gamma} of \Cref{theo:petrov_6_simplify} hold for some $\tau_0\in C_+(M_0)$. Assume that the following \eqref{enu:theo_petrov_infty_levy_meas} and \eqref{enu:theo_petrov_infty_sigma_bounded} hold.
	\begin{enumerate}[(a)]
		\item\label{enu:theo_petrov_infty_levy_meas} There is a dense subset $\mathcal D$ of $(0,\infty)$ and a measure $M_0$ on $(0,\infty]$ 
		with 
		\begin{align*}
		\sum_{i=1}^{k_n} P\left(  Y_{k_n,i} > x \right)\;  \to \;\, M_0 (x,\infty]\in\R \text{ for all } x\in {\mathcal D }. 
		\end{align*}
		
		\item \label{enu:theo_petrov_infty_sigma_bounded} There exists some $\tau_1>0$ such that
		\begin{align*}
		\limsup_{n\to\infty}\sum_{i=1}^{k_n} \int_{\left\{\left| Y_{k_n,i} \right|<\tau_1\right\}} Y_{k_n,i}^2\,\mathrm{ d }  P< \infty.
		\end{align*}
	\end{enumerate}
	Then,
	\begin{align*}
	{\mathcal L }\Bigl( \sum_{i=1}^{k_n}Y_{n,i} \Bigr) \overset{\mathrm w}{\longrightarrow} e^{-M_0(\{\infty\})}\nu +(1-e^{-M_0(\{\infty\})})\epsilon_{\infty},
	\end{align*}
	where $\nu$ is a infinitely divisible measure on $\R$ with L\'{e}vy-Khintchine triplet $(\gamma,\sigma^2,\eta)$ and L\'{e}vy measure $\eta=M_{0|(0,\infty)}$.
\end{theorem}
\begin{proof}
	Put $\eta=M_{0|(0,\infty)}$. Let the sequence $(M_n)_{n\in\N}$ consists of measures on $(0,\infty]$  given by $M_n(x,\infty]= \sum_{i=1}^{k_n}P(Y_{n,i}>x)$, $x>0$. Clearly, $M_{n|(0,\infty)}\overset{\mathrm w}{\longrightarrow}\eta$ and $\limsup_{n\to\infty} \int_{(0,\tau_1)} t^2 \,\mathrm{ d }M_n(t)<\infty$. Thus, we obtain $\int \min(t^2,1) \,\mathrm{ d }\eta(t)<\infty$, which proves that $\eta$ is a L\'{e}vy measure. Define $Z_{n,u}=\sum_{i=1}^{k_n}Y_{n,i} \mathbf{1}\{Y_{n,i}\leq u\}$ for all $u\in{\mathcal D }$, $u>\tau_0$. 
	By \Cref{theo:petrov_6_simplify} $Z_{n,u}$ converges in distribution to $X_u$, where $X_u$ is infinitely divisible with L\'{e}vy-Khintchine triplet $(\gamma_u,\sigma^2,\eta_u)$, L\'{e}vy measure $\eta_u=\eta_{(0,u]}$ and shift term
	\begin{align*}
	\gamma_u=\gamma - \int_{(u,\infty)}\frac{x}{1+x^2}\,\mathrm{ d }\eta(x).
	\end{align*}
	Since $\eta$ is L\'{e}vy measure it is easy to verify $\gamma_u\to \gamma$ as ${\mathcal D }\ni u \to\infty$. By this and  Theorem 3.19.2 of \cite{gndedenKolmogorov} $X_u$ converges in distribution to $X$ as ${\mathcal D }\ni u \to\infty$, where $X\sim \nu$.
	Now, let $(u_n)_{n\in\N}$ be a sequence in ${\mathcal D }$ which tends to $\infty$ slowly enough such that
	$\sum_{i=1}^{k_n} P\left(  Y_{k_n,i} > u_n \right)\;  \to \;\, M_0(\{\infty\}).$ 
	Standard arguments, see Theorem 3.2 of \citet{billingsley1999}, imply that $Z_{n,u_n}$ converges in distribution to $X$ since
	for all $\delta>0$ 
	\begin{align*}
	\limsup_{n\to\infty} P\Bigl( \Bigl | Z_{n,u}-Z_{n,u_n} \Bigr|\geq \delta \Bigl)\leq M_0(u,\infty)  \to 0 \text{ as }{\mathcal D }\ni u \to \infty.
	\end{align*}
	The basic idea to determine the limit distribution of $\sum_{i=1}^{k_n}Y_{n,i}$ is to condition on $C_n=\{\max_{1\leq i \leq k_n}Y_{n,i}\leq u_n\}$. Note that for all $t\in\R$
	\begin{align*}
	P\Bigl( \sum_{i=1}^{k_n}Y_{n,i}\leq t \Bigr)=P( Z_{n,u_n}\leq t | C_n )P(C_n) + P\Bigl ( \sum_{i=1}^{k_n}Y_{n,i}\leq t , \max_{1\leq i \leq k_n}Y_{n,i}> u_n \Bigr),
	\end{align*}
	where the latter summand tends to $0$. Moreover, observe that
	\begin{align*}
	1-P(C_n) = \prod_{i=1}^{k_n}\Bigl( 1-P(Y_{n,i}> u_n) \Bigr) \to e^{-M_0(\{\infty\})}.
	\end{align*}
	It is remains to show that $Z_{n,u_n}$ tends to $X$ conditioned on $C_n$. Conditioned on $C_n$ we have $Z_{n,u_n}=\sum_{i=1}^{k_n} Y_{n,i}\mathbf{1}\{Y_{n,i}\leq u_n\}$  and $(Y_{n,i}\mathbf{1}\{Y_{n,i}\leq u_n\})_{i\leq k_n}$ is a rowwise independent and infinitesimal triangular array. Hence, we can apply \Cref{theo:petrov_6_simplify} to $Z_{n,u_n}$ conditioned on $C_n$. Finally, by basic calculations   \Cref{theo:petrov_6_simplify}\eqref{enu:petrov_6_simp_levy_mass}-\eqref{enu:petrov_6_simp_levy_gamma} are fulfilled for the same $\eta$, $\sigma^2$ and $\gamma$ given by the L\'{e}vy-Khintchine triplet of the limit X of $Z_{n,u_n}$, e.g. we have for all $x\in {\mathcal D }$
	\begin{align*}
	\sum_{i=1}^{k_n}P\Bigl(Y_{n,i}\mathbf{1}\{Y_{n,i}\leq u_n\}> x |C_n\Bigr) = \sum_{i=1}^{k_n} \frac{P(Y_{n,i}>x)-P(Y_{n,i}>u_n)}{P(Y_{n,i}\leq u_n)} \to \eta(x,\infty)
	\end{align*}
	since $\min_{1\leq i \leq k_n} P(Y_{n,i}\leq u_n)\geq 1-\max_{1\leq i \leq k_n}P(Y_{n,i}\geq 1)\to 1$. 
\end{proof}

\subsection{Proofs of \Cref{sec:LLRT} and \Cref{appendix:additional}}
\subsubsection{Proof of \Cref{theo:trivial_limits}}
The statement of \Cref{theo:trivial_limits} follows immediately from the following lemma.

\begin{lemma}\label{lem:trivial_limits}
	Let  $I_{n,1,x}$ and $I_{n,2,x}$, $x>0$, be defined as in in \eqref{eqn:I1} and \eqref{eqn:I2}, respectively.  Let $D_n$ be defined as in \eqref{eqn:DN}. Then we have for all $\tau>0$:
	\begin{enumerate}[(a)]
		\item\label{enu:lem_trivial_limits_abschaet} There exists a constant $C_\tau>0$ such that
		\begin{align}	
		&D_n	\leq \Bigl( \frac{1}{2}+\max_{1\leq i \leq k_n}\varepsilon_{n,i} \Bigr) I_{n,1,\tau} + I_{n,2,\tau} \label{eqn:boundaries_Dn_upper}	\\
		&D_n \geq 		C_\tau\max\Bigl\{I_{n,1,\tau},I_{n,2,\tau}-\frac{2}{\tau}I_{n,1,\tau}\max_{1\leq i \leq k_n}\varepsilon_{n,i}\Bigr\}. \label{eqn:boundaries_Dn_lower}	
		\end{align}
	\end{enumerate}
\end{lemma}
\begin{remark}\label{rem:}
	The idea and the proof of the upper bound of $D_n$ in \eqref{eqn:boundaries_Dn_upper} is based on the argumentation of \citet{CaiJengJin2011} on pp. 21f.
\end{remark}
\begin{proof}[Lemma \ref{lem:trivial_limits}]
	To shorten the notation, we define 
	\begin{align}\label{eqn:def_A} 
	A_{n,i,x}=\Bigl\{\varepsilon_{n,i}\frac{\mathrm{ d } \mu_{n,i}}{\,\mathrm{ d }P_{n,i}} > x\Bigr\}\textrm{ for all }x>0.
	\end{align}
	We can deduce from \eqref{eqn:def_hellinger} that
	\begin{align}\label{eqn:cor_uniform_suff_B2_B3}
	D_{n} \,\leq\; \sum_{i=1}^{k_n}\ew_{P_{n,i}} \left( 1-\:\sqrt{1-\varepsilon_{n,i} + \varepsilon_{n,i} \frac{\mathrm{ d } \mu_{n,i}}{\,\mathrm{ d }P_{n,i}}\mathbf{1}(A_{n,i,\tau}^c)}\: \right).
	\end{align}
	Note that $1-\sqrt{1+t}\leq -t/2+t^2$ for all $t\geq -1$. Applying this (pointwisely) to the integrand in \eqref{eqn:cor_uniform_suff_B2_B3} with $t=\varepsilon_{n,i}(\frac{\mathrm{ d } \mu_{n,i}}{\,\mathrm{ d }P_{n,i}}\mathbf{1}(A_{n,i,\tau}^c)-1)$ yields \eqref{eqn:boundaries_Dn_upper}. \\
	We split the proof of \eqref{eqn:boundaries_Dn_lower} into two steps. First, define for all $x>0$
	\begin{align}\label{eqn:def_I2tilde}
	\widetilde{I}_{n,2,x}&=\sum_{i=1}^{k_n}\int_{A_{n,i,x}^c } \varepsilon_{n,i}^2\Bigl( \frac{\mathrm{ d } \mu_{n,i}}{\,\mathrm{ d }P_{n,i}}-1 \Bigr)^2\,\mathrm{ d }P_{n,i}.
	\end{align}
	For $\varepsilon_{n}^{\text{max}}=\max_{1\leq i \leq k_n}\varepsilon_{n,i}$ we can deduce from $\varepsilon_{n}^{\text{max}}\geq(\varepsilon_{n}^{\text{max}})^2$ and
	\begin{align}
	&\sum_{i=1}^{k_n} P_{n,i}(Y_{n,i}>x) \leq \sum_{i=1}^{k_n} P_{n,i}(A_{n,i,e^x-1}) \leq \frac{1}{e^x-1} I_{n,1,e^x-1}\label{eqn:P(A^c)<I}\\
	&\text{that }	-\frac{2\varepsilon_{n}^{\text{max}}}{x}I_{n,1,x}\leq\widetilde{I}_{n,2,x}-I_{n,2,x}\leq 2\varepsilon_{n}^{\text{max}} I_{n,1,x}\label{eqn:bounds_I2tilde}
	\end{align} 
	for all $x>0$. Since $\mathrm{ d } Q_{n,i}/\mathrm{ d }P_{n,i}$ is bounded from above by $1+\tau$  on $A_{n,i,\tau}^c$ we obtain
	\begin{align*}
	2 D_n \;\geq\; \sum_{i=1}^{k_n} \int  \Bigl(1-\frac{\mathrm{ d } Q_{n,i}}{\,\mathrm{ d }P_{n,i}}\Bigr)^2\Bigl(1+\Bigl(\frac{\mathrm{ d } Q_{n,i}}{\,\mathrm{ d }P_{n,i}}\Bigr)^{1/2}\Bigr)^{-2} \mathbf{1}(A_{n,i,\tau}^c) \,\mathrm{ d }P_{n,i} \geq \frac{\widetilde I_{n,2,\tau}}{(1+\sqrt{{1+\tau}})^2} .
	\end{align*}
	Combining this and \eqref{eqn:bounds_I2tilde} gives us the first bound in \eqref{eqn:def_I2tilde} for appropriate $C_\tau$. Second, set $C= 1/(\sqrt{\tau/2+1} + 1 )<1/2$. Note that on $A_{n,i,\tau}$ 
	\begin{align*}
	\Bigl( \frac{\mathrm{ d } Q_{n,i}}{\,\mathrm{ d }P_{n,i}} \Bigr)^{1/2} - 1    = \Bigl( \frac{\mathrm{ d } Q_{n,i} }{\,\mathrm{ d }P_{n,i}} - 1 \Bigr)\Bigl( \Bigl( \frac{\mathrm{ d } Q_{n,i}}{\,\mathrm{ d }P_{n,i}} \Bigr)^{1/2} + 1 \Bigr)^{-1} \leq C\left( \frac{\mathrm{ d } Q_{n,i}}{\,\mathrm{ d }P_{n,i}} - 1 \right).
	\end{align*}
	Consequently,
	\begin{align*}
	2\:D_n & \geq \sum_{i=1}^{k_n} E_{P_{n,i}}\Bigl( \Bigl( \frac{\mathrm{ d } Q_{n,i}}{\,\mathrm{ d }P_{n,i}}- 1 - 2	\Bigr( \Bigl(\frac{\mathrm{ d } Q_{n,i}}{\,\mathrm{ d }P_{n,i}}\Bigr)^{1/2} - 1 \Bigr)\Bigr)\mathbf{1}(A_{n,i,\tau})\Bigr)\nonumber \\
	& \geq ( 1 - 2C) \left( 1- \frac{\max_{1\leq i \leq k_n}\varepsilon_{n,i}}{\tau} \right) \sum_{i=1}^{k_n} \varepsilon_{n,i} \mu_{n,i}\left( A_{n,i,\tau} \right).
	\end{align*}
	Finally, \eqref{enu:lem_trivial_limits_abschaet} is shown and combining it with Lemma \ref{lem:D_detection} yields \eqref{enu:lem_trivial_limits_comp_det} and \eqref{enu:lem_trivial_limits_undet}.
	
\end{proof}

\subsubsection{Proof of \Cref{theo:xi_real_oder_xi_fullinfo}\eqref{enu:theo:xi_real_oder_xi_fullinfo_infinit_divi}}
The statements follows from Remark (8.6) and Lemma (8.7) of \citet{JanssenMilbrodtStrasser1985} as we explain in the following. Let $C^2_{lok}(\R)$ be set of all bounded functions $f:\R\to\R$ that are twice differentiable with continuous derivatives in some neighbourhood of $0$. Denote by $f^{(k)}(0)$ the \textit{k}th derivative of $f$ at $0$. The L\'{e}vy-Khintchine triplet of a infinitely divisible measure $\nu$ is equal to  $(\gamma,\sigma^2,\eta)$ if and only if the \textit{generating functional} $A:C^2_{lok}(\R)\to \R$ admits the L\'{e}vy-Khintchine representation
\begin{align*}
A(f) = f^{(1)}(0)\gamma +\sigma^2 f^{(2)}(0) + \int_{\R\setminus\{0\}} \Bigl( f(x)-f(0)- \frac{f^{(1)}(0)x}{1+x^2} \Bigr)\,\mathrm{ d }\eta(x)
\end{align*}
for all $f\in C^2_{lok}(\R)$. For the actual definition of $A$ and more details about it we refer the reader to \citet{JanssenMilbrodtStrasser1985}, in particular to (8.1)-(8.4).
\begin{lemma}\label{lem:xi1+2_real_levy_chara}
	Let $\{\widetilde \nu_1,\widetilde{\nu}_2\}$ be some binary experiment in its standard form, compare to \eqref{eqn:bin_conv_def_nu1} and \eqref{eqn:bin_conv_def_nu2}, such that $\widetilde{ \nu_1}(\R)=\widetilde\nu_2(\R)=1$ and $\widetilde\nu_1$ is infinitely divisible with L\'{e}vy-Khintchine triplet $(\gamma,\sigma^2,\eta)$. Then $\widetilde\nu_2$ is also infinitely divisible with L\'{e}vy-Khintchine triplet $(\gamma_2,\sigma^2_2,\eta_2)$, where  $\sigma_1^2=\sigma^2_2$,  $\eta_2\ll \eta_1$ with Radon-Nikodym derivative $x\mapsto\mathrm{ d } \eta_2/\mathrm{ d }\eta_1(x)=e^x$ and 
	\begin{align}
	&\gamma_1 +\frac{\sigma^2_1}{2} - \int \Bigl( 1-e^x + \frac{x}{x^2+1} \Bigr) \,\mathrm{ d } \eta_1(x)=0,\label{eqn:gamma1_appendix}\\
	&\gamma_2=\gamma_1 + \sigma_1^2 + \int ( e^x-1)\frac{x}{1+x^2} \,\mathrm{ d }\eta_1(x).\label{eqn:gamma2_appendix}
	\end{align}
\end{lemma}
\begin{remark}\label{rem:xi1+2_int_finite}
	Since $\int x^2\mathbf{1}(|x|\leq 1) \,\mathrm{ d }\eta_1(x), \int e^x \mathbf{1}(|x|\geq 1) \,\mathrm{ d }\eta_1(x)<\infty$, see Lemma (8.7)(a) of \cite{JanssenMilbrodtStrasser1985}, the integrals in \eqref{eqn:gamma1_appendix} and \eqref{eqn:gamma2_appendix} are finite.
\end{remark}
\begin{proof}[Lemma \ref{lem:xi1+2_real_levy_chara}]
	Let $A$ be the generating functional of $\widetilde\nu_1$. Combining $\int \exp \,\mathrm{ d }\widetilde\nu_1=\widetilde\nu_2(\R)=1$ and Lemma (8.7)(b) and (c) from \cite{JanssenMilbrodtStrasser1985} we deduce that $A(\exp)=0$ and $C^2_{lok}(\R)\ni f\mapsto A(\exp f)$ is the generating functional of $\widetilde\nu_2$ and, in particular, $\widetilde \nu_2$ is infinitely divisible. Using the L\'{e}vy-Khintchine representation of $A$ immediately yields that $A(\exp)$ is equal to the left side of \eqref{eqn:gamma1_appendix}, which proves \eqref{eqn:gamma1_appendix}. From $f(0)A(\exp)=0$ we get  for all $f\in C^2_{lok}(\R)$
	\begin{align*}
	A(f\exp)&=	f^{(1)}(0) \Bigl( \gamma_1 + \sigma^2_1 +  \int ( e^x-1)\frac{x}{1+x^2} \,\mathrm{ d }\eta_1(x)  \Bigr) \\
	&\phantom{=}+ f^{(2)}(0)\frac{\sigma^2_1}{2} + \int \Bigl( f(x) - f(0) - \frac{f^{(1)}(0)x}{1+x^2} \Bigr) e^x \,\mathrm{ d } \eta_1(x).	
	\end{align*}
	Consequently, the statements about $(\gamma_2,\sigma_2^2,\eta_2)$ follow.
\end{proof}
Now, we prove \Cref{theo:xi_real_oder_xi_fullinfo}\eqref{enu:theo:xi_real_oder_xi_fullinfo_infinit_divi}. Since $\frac{\mathrm{ d } Q_{n,i}}{\,\mathrm{ d }P_{n,i}} \geq 1-\max_{1\leq i \leq k_n}\varepsilon_{n,i} \to 1$ 
\eqref{eqn:petrov_6_simp_cond_emptyset} is fulfilled and by \Cref{theo:petrov_6_simplify} $\eta_1$ is concentrated on $(0,\infty)$. Now, consider $\{\widetilde{\nu}_1,\widetilde{\nu}_2\}=\{\nu_1*\epsilon_{-\log(a)},(a^{-1}\nu_{2|\R})*\epsilon_{-\log(a)}\}$. This binary experiment is in its standard from since
\begin{align*}
\frac{\,\mathrm{ d }\widetilde{ \nu}_2}{\,\mathrm{ d }\widetilde{ \nu}_1}(x)= a^{-1} \frac{ \,\mathrm{ d }\nu_2 }{ \,\mathrm{ d }\nu_1 }(x+\log(a))=\exp(x),\,x\in\R.
\end{align*}
Clearly, $\widetilde \nu_1$ is infinitely divisible with L\'{e}vy characteristic $(\gamma_1-\log(a), \sigma_1^2,\eta_1)$ and $\widetilde{\nu}_1(\R)=\widetilde{ \nu}_2(\R)=1$. Applying Lemma \ref{lem:xi1+2_real_levy_chara} proves that $\widetilde \nu_2$ is infinitely divisible and so is $\rho=a^{-1}\nu_{2|\R}$. Moreover, is easy to check that we obtain all statements about the L\'{e}vy-Khintchine triplets.

\subsubsection{Proof of \Cref{theo:general_limit_theorem}}
We carried out two different proofs for \Cref{theo:general_limit_theorem}. The first one relies on infinitely divisible statistical experiments and accompanying Poisson experiments, and arguments from Chap. 4, 5, 9, 10 of \citet{JanssenMilbrodtStrasser1985} are used. The second one is based on traditional limit theorems for real-valued random variables. Since, probably, the second one is easier to follow for  the readers who are not experts in the field of statistical experiments we decided to present only the second proof. \\
At the end of the proof we will verify the following lemma.
\begin{lemma}\label{lem:general_limit_theorem}
	Suppose that \eqref{enu:theo:general_limit_theorem_levy_mass} and \eqref{enu:theo:general_limit_theorem_sigma} hold. Then the sums in \Cref{theo:petrov_6_simplify} \eqref{enu:petrov_6_simp_levy_sigma} and \eqref{enu:petrov_6_simp_levy_gamma} and in \Cref{theo:petrov_including_infinity}\eqref{enu:theo_petrov_infty_levy_meas} and \eqref{enu:theo_petrov_infty_sigma_bounded} for $Y_{n,i}$ defined by 
	\begin{align}\label{eqn:def_Y_general_limit}
	Y_{n,i}=\log\frac{\mathrm{ d } Q_{n,i}}{\,\mathrm{ d }P_{n,i}}
	\end{align}
	are upper bounded for every $x>0$ and all sufficiently small $\tau_0,\tau_1\in {\mathcal D }$, respectively, under $P_{(n)}$ as well as under $Q_{(n)}$. In particular,  \Cref{theo:petrov_6_simplify}\eqref{enu:petrov_6_simp_levy_sigma} is fulfilled for $\sigma^2$ under $P_{(n)}$. 
\end{lemma} 
Let us first assume that \eqref{enu:theo:general_limit_theorem_levy_mass} and \eqref{enu:theo:general_limit_theorem_sigma} are fulfilled. Define $Y_{n,i}$ as in \eqref{eqn:def_Y_general_limit}. Regarding Lemma \ref{lem:general_limit_theorem} and using typical sub-subsequence arguments we can assume without loss of generality that \Cref{theo:petrov_6_simplify}\eqref{enu:petrov_6_simp_levy_mass} and \eqref{enu:petrov_6_simp_levy_sigma} as well as \Cref{theo:petrov_including_infinity}\eqref{enu:theo_petrov_infty_levy_meas} and \eqref{enu:theo_petrov_infty_sigma_bounded} hold for a measure $M_1$ (resp. $M_2$), $\sigma_1\geq 0$ ($\sigma_2\geq 0$, resp.) and $\gamma_1\in\R$ ($\gamma_2\in\R$, resp.) under $P_{(n)}$ ($Q_{(n)}$, resp.). In particular, by Lemma \ref{lem:general_limit_theorem} $\sigma_1^2=\sigma^2$. Note that $\eta_j=M_{j|(0,\infty)}$ is a L\'{e}vy measure. From \eqref{eqn:P(A^c)<I}
we obtain $M_1(\{\infty\})=0$ and so $\xi_1$, the limit of $T_n$ under $P_{(n)}$, is real-valued. Moreover, since $\max_{1\leq i \leq k_n}\varepsilon_{n,i}\to 0$ and $\varepsilon_{n,i}\mu_{n,i}( A_{n,i,e^x-1+\varepsilon_{n,i}}) =  Q_{n,i}(Y_{n,i}>x) -  (1-\varepsilon_{n,i})P_{n,i}(Y_{n,i}>x)$
we can deduce that $M_{|(0,\infty)}=\eta_2-\eta_1$ and $M_2(\{\infty\})= M(\{\infty\})$.
Finally, the proof for the first assertion is completed by \Cref{theo:xi_real_oder_xi_fullinfo}\eqref{enu:theo:xi_real_oder_xi_fullinfo_infinit_divi}.\\
Now, let $\xi_1$ be not equal to $-\infty$ with probability one. By \Cref{theo:trivial_limits}\eqref{enu:lem_trivial_limits_comp_det} we have  $\sup_{n\in\N} I_{n,1,\tau}+I_{n,2,\tau}<\infty$ for all $\tau>0$. Hence, for each subsequence there is a subsequence such that \eqref{enu:theo:general_limit_theorem_levy_mass} for some measure $M$ and \eqref{enu:theo:general_limit_theorem_sigma} for some $\sigma^2$ are fulfilled. From \Cref{theo:xi_real_oder_xi_fullinfo}\eqref{enu:theo:xi_real_oder_xi_fullinfo_infinit_divi}  and the first assertion proved above we obtain: $\xi_1$ is real-valued, and $M$ and $\sigma^2$ are uniquely determined by the distribution of $\xi_1$ and so do not depend on the special choice of the subsequence, which proves the second assertion (and  \Cref{theo:xi_real_oder_xi_fullinfo}\eqref{enu:theo:xi_real_oder_fullinfo}).

\begin{proof}[Lemma \ref{lem:general_limit_theorem}]
	First, observe that by \eqref{eqn:P(A^c)<I} the sum in \Cref{theo:petrov_including_infinity}\eqref{enu:theo_petrov_infty_levy_meas} is upper bounded under $P_{(n)}$ as well as under $Q_{(n)}$ for all $\tau>0$.  By \eqref{eqn:maxeps_to_0} 
	\begin{align}\label{eqn:def_B}
	B_{n,i,\tau}=\left\{ \left| Y_{n,i} \right|\leq \tau \right\} = A_{n,i,t_{n,i}(\tau)}^c
	\end{align}
	if $n\geq N_\tau$ is sufficiently large, where $t_{n,i}(\tau)=e^\tau-1+\varepsilon_{n,i}\in[e^\tau-1, e^\tau]$. Define $\widetilde I_{n,2,x}$ as in \eqref{eqn:def_I2tilde}.
	By Taylor's formula there exists some random variable $R_{n,i,\tau}$ with $R_{n,i,\tau}=0$ on $B^c_{n,i,\tau}$ such that we have on $B_{n,i,\tau}$
	\begin{align}\label{eqn:repre_log}
	Y_{n,i}= \varepsilon_{n,i}\Bigl( \frac{\mathrm{ d } \mu_{n,i}}{\,\mathrm{ d }P_{n,i}}-1 \Bigr) - \varepsilon_{n,i}^2\Bigl( \frac{\mathrm{ d } \mu_{n,i}}{\,\mathrm{ d }P_{n,i}}-1 \Bigr)^2 \Bigl( \frac{1}{2} + R_{n,i,\tau} \Bigr) 
	\end{align}
	and $\max_{1\leq i \leq k_n}| R_{n,i,\tau}|\leq C_\tau$ for some constant $C_\tau\in(0,\infty)$ with $C_\tau\to 0$ as $\tau \searrow 0$.  Combining this and \eqref{eqn:P(A^c)<I} yields 
	\begin{align*}
	\Bigl | \sum_{i=1}^{k_n}\int_{B_{n,i,\tau} }Y_{n,i}\,\mathrm{ d }P_{n,i} \Bigr| \leq \Bigl( 1+ \frac{1}{e^\tau-1}\Bigr) I_{n,1,e^\tau-1} + \Bigl(\frac{1}{2}+C_\tau\Bigr)\widetilde I_{n,2,e^\tau},
	\end{align*}
	where by \eqref{eqn:bounds_I2tilde} the upper bound is bounded itself for all sufficiently small $\tau>0$. Since $Q_{n,i}=(1-\varepsilon_{n,i})P_{n,i}+\varepsilon_{n,i}\mu_{n,i}$ and $\frac{\mathrm{ d } \mu_{n,i}}{\,\mathrm{ d }P_{n,i}}\leq e^\tau$ on $B_{n,i,\tau}$ we obtain similarly the following upper bound of $| \sum_{i=1}^{k_n}\int_{B_{n,i,\tau} }Y_{n,i}\,\mathrm{ d }Q_{n,i} |$:
	\begin{align*}
	\Bigl | \sum_{i=1}^{k_n}\int_{B_{n,i,\tau} }Y_{n,i}\,\mathrm{ d }P_{n,i} \Bigr| + I_{n,1,e^\tau-1}
	+ \Bigl(1+\Bigl(\frac{1}{2}+C_\tau\Bigr)e^\tau\Bigr)\widetilde I_{n,2,e^\tau},
	\end{align*}
	which itself is bounded for all small $\tau>0$, see also \eqref{eqn:bounds_I2tilde}. In the last step we discuss the sum in \Cref{theo:petrov_6_simplify}\eqref{enu:petrov_6_simp_levy_sigma}. On $B_{n,i,\tau}$ we obtain the following inequalities from \eqref{eqn:repre_log} for all sufficiently small $\tau>0$ such that $C_\tau\leq \frac{1}{2}$:
	\begin{align*}
	\varepsilon_{n,i}\Bigl| \frac{\mathrm{ d } \mu_{n,i}}{\,\mathrm{ d }P_{n,i}}-1 \Bigr| \left( 2- e^\tau - 2 \varepsilon_{n,i} \right)\leq 	|Y_{n,i} | \leq \varepsilon_{n,i}\Bigl| \frac{\mathrm{ d } \mu_{n,i}}{\,\mathrm{ d }P_{n,i}}-1 \Bigr| \left( e^\tau + 2 \varepsilon_{n,i} \right).
	\end{align*}
	From this, \eqref{eqn:bounds_I2tilde} and $\frac{\mathrm{ d } Q_{n,i}}{\,\mathrm{ d }P_{n,i}}\leq e^\tau+\max_{1\leq i \leq k_n}\varepsilon_{n,i}$ on $B_{n,i,\tau}$ we conclude
	\begin{align*}
	\lim_{\tau\searrow 0}\; \underset{n\to\infty}{ \substack{ \limsup \\ \liminf } }\sum_{i=1}^{k_n}\int_{B_{n,i,\tau}} Y_{n,i}^2 \,\mathrm{ d }Q_{n,i}\leq\lim_{\tau\searrow 0}\; \underset{n\to\infty}{ \substack{ \limsup \\ \liminf } }\sum_{i=1}^{k_n}\int_{B_{n,i,\tau}} Y_{n,i}^2 \,\mathrm{ d }P_{n,i}=\sigma^2.
	\end{align*}
	Since $(a+b)^2\leq 4a^2+4b^2$ we have for all sufficiently small $\tau>0$ that
	\begin{align*}
	&\frac{1}{4}\sum_{i=1}^{k_n} \Bigl( \int_{B_{n,i,\tau} }Y_{n,i}  \,\mathrm{ d } P_{n,i}  \Bigr)^2\\
	&\leq  \sum_{i=1}^{k_n} \Bigl( \varepsilon_{n,i}\int_{B_{n,i,\tau}^c}  1-\frac{\mathrm{ d } \mu_{n,i}}{\,\mathrm{ d }P_{n,i}}   \,\mathrm{ d } P_{n,i}\Bigr)^2 +\Bigl( \int_{B_{n,i,\tau}} \varepsilon_{n,i}^2 \Bigl( \frac{\mathrm{ d } \mu_{n,i}}{\mathrm{ d }P_{n,i}}-1 \Bigr)^2  \,\mathrm{ d }P_{n,i}	\Bigr)^2\nonumber \\
	&\leq  \Bigl( \max_{1\leq i \leq k_n} \varepsilon_{n,i} \Bigr)(1+\tau ) I_{n,1,e^{\tau}-1} +\widetilde I_{n,2,e^\tau} (e^\tau - 1 + \max_{1\leq i \leq k_n}\varepsilon_{n,i})^2.\nonumber 
	\end{align*}
	Hence, by \eqref{eqn:bounds_I2tilde} $\lim_{\tau\searrow 0} \limsup_{n\to\infty}(\int_{B_{n,i,\tau}} Y_{n,i} \,\mathrm{ d }P_{n,i})^2=0$ and, consequently, \Cref{theo:petrov_6_simplify}\eqref{enu:petrov_6_simp_levy_sigma} is fulfilled for $\sigma^2$ under $P_{(n)}$.
\end{proof}

\subsubsection{Proof of \Cref{theo:xi_real_oder_xi_fullinfo}\eqref{enu:theo:xi_real_oder_fullinfo}}
We verified  \Cref{theo:xi_real_oder_xi_fullinfo}\eqref{enu:theo:xi_real_oder_fullinfo} while proving the second assertion of \Cref{theo:general_limit_theorem}.
´\subsubsection{Proof of \Cref{theo:normal_limits}}
The equivalence of \eqref{enu:theo:normal_limits_both_gauss}-\eqref{enu:theo:normal_limits_xi2_normal} follows from \eqref{eqn:relation_nu_2_nu_1} and is standard for binary experiments, see \citet{Strasser1985}. The equivalence of  \eqref{enu:theo:normal_limits_xi2_maxq} and \eqref{enu:theo:normal_limits_xi2_maxmu} follows from \eqref{eqn:def_qni} and \eqref{eqn:maxeps_to_0}. Define $A_{n,i,x}$  as in \eqref{eqn:def_A}.

\noindent \underline{Equivalence of \eqref{enu:theo:normal_limits_xi1} and \eqref{enu:theo_normal_limits_my_cond}:} By \Cref{theo:general_limit_theorem} $I_{n,1,x}\to 0$  for all $x>0$ holds also under \eqref{enu:theo:normal_limits_xi1}. Hence, we can suppose that this convergence is fulfilled subsequently. Fix $\tau>0$. Then
\begin{align*} 
&0 \leq  E_{P_{n,i}}\left(  \varepsilon_{n,i}^2 \Bigl( \frac{\mathrm{ d } \mu_{n,i}}{\,\mathrm{ d }P_{n,i}} \Bigr)^2 \mathbf{1}\Bigl\{ \frac{\mathrm{ d } \mu_{n,i}}{\,\mathrm{ d }P_{n,i}}\in(x,\tau] \Bigr\}  \right)  \leq  \tau\:  \varepsilon_{n,i} \mu_{n,i}( A_{n,i,x}  ) 
\end{align*}
holds for all $x\in(0,\tau]$  and so $I_{n,2,x} - I_{n,2,\tau}\to 0$ does.
Consequently, \eqref{enu:theo_normal_limits_my_cond} holds if and only if \Cref{theo:general_limit_theorem}\eqref{enu:theo:general_limit_theorem_levy_mass} and \eqref{enu:theo:general_limit_theorem_sigma} do so for the same $\sigma^2\in[0,\infty)$ and $M\equiv 0$. Hence,  the equivalence of \eqref{enu:theo:normal_limits_xi1} and \eqref{enu:theo_normal_limits_my_cond} follows from \Cref{theo:general_limit_theorem}.

\noindent \underline{Equivalence of \eqref{enu:theo:normal_limits_Zn} and \eqref{enu:theo_normal_limits_my_cond}:}
Define $Y_{n,i}$  as in \eqref{eqn:def_Y_general_limit} and set $\widetilde Y_{n,i}= f(Y_{n,i})$ for $f(x)=\exp(x)-1$, $x\in\R$. Note that $f(0)=0$ and $f'(0)=f''(0)=1$. From this, a Taylor expansion, compare to \eqref{eqn:repre_log}, and \Cref{theo:petrov_6_simplify} we obtain  that $\sum_{i=1}^{k_n}Y_{n,i}$ converges in distribution to $X$ with L\'{e}vy-Khintchine triplet $(0,\sigma^2,0)$ if and only if $\sum_{i=1}^{k_n}\widetilde Y_{n,i}$ does so to $\widetilde X$ with L\'{e}vy-Khintchine triplet $(-\sigma^2/2,\sigma^2,0)$.

\noindent\underline{Equivalence of \eqref{enu:theo:normal_limits_xi2_real_xi1_normal} and \eqref{enu:theo:normal_limits_xi2_maxmu}:} Throughout this proof step we can assume that $\xi_2$ is real-valued and so is $\xi_1$, see \Cref{theo:xi_real_oder_xi_fullinfo}\eqref{enu:theo:xi_real_oder_fullinfo}. By the first Lemma of Le Cam $P_{(n)}$ and $Q_{(n)}$ are mutually contiguous, see also Remark \ref{rem:contiguous}. Hence, \eqref{enu:theo:normal_limits_xi2_maxmu} is true if and only if for all $x>0$
\begin{align*}
0 \leftarrow Q_{(n)}\Bigl( \max_{1\leq i \leq k_n}\varepsilon_{n,i}\frac{\mathrm{ d } \mu_{n,i}}{\,\mathrm{ d }P_{n,i}}>x \Bigr)= 1-\prod_{i=1}^{k_n} \Bigl( 1-Q_{n,i}( A_{n,i,x} ) \Bigr).
\end{align*}
Combining this and \eqref{eqn:P(A^c)<I} yields that \eqref{enu:theo:normal_limits_xi2_maxmu} is fulfilled if and only if $I_{n,1,x}\to 0$ for all $x>0$. Finally, note that  $\xi_1$ is normal distributed if and only if it has trivial L\'{e}vy measure $\eta_1\equiv 0$, which by \Cref{theo:general_limit_theorem} is true if and only if $I_{n,1,x}\to 0$ for all $x>0$.

\subsubsection{Proof of Lemma \ref{lem:LAN_var_not_to_sigma_REPLACE}}
Let $\widetilde Q_{n,i}$ and $\widetilde Q_{(n)}$ be defined as $Q_{n,i}$ and $Q_{(n)}$ replacing $\mu_{n,i}$ and $\varepsilon_{n,i}$ by $\widetilde \mu_{n,i}$ and $\widetilde \varepsilon_{n,i}$. For the statement in Lemma \ref{lem:LAN_var_not_to_sigma_REPLACE} it is sufficient to show that $\{Q_{(n)},\widetilde Q_{n,i}\}$ tend weakly to the uninformative experiment $\{\epsilon_0,\epsilon_0\}$. The main task for this purpose is to verify $\sum_{i=1}^{k_n} || Q_{n,i}(\boldsymbol{\theta})-Q_{n,i}(\boldsymbol{\widetilde\theta})|| \to 0$, which is left to the reader.

\subsubsection{Proof of \Cref{theo:ARE}}
Denote by $Z_n(\boldsymbol{\theta_1})$ and $Z_n(\boldsymbol{\theta_2})$ the statistic introduced in \eqref{eqn:def_Zn} for the model $\boldsymbol{\theta_1}$ and $\boldsymbol{\theta_2}$, respectively. Since these statistics are linear, the multivariate central limit theorem implies distributional convergence $(Z_n(\boldsymbol{\theta_1}),Z_n(\boldsymbol{\theta_2}))\overset{\mathrm d}{\longrightarrow} \widetilde Z \sim N ( (0,0), (\gamma(\boldsymbol{\theta_i},\boldsymbol{\theta_j}))_{1\leq i,j\leq 2} )$ under $P_{(n)}$. In the next step we verify for $j=1,2$
\begin{align}\label{eqn:LAN}
T_n(\boldsymbol{\theta_j})  = Z_n(\boldsymbol{\theta_j})- \frac{\gamma(\boldsymbol{\theta_j},\boldsymbol{\theta_j})}{2} + R_{n,j}, 
\end{align}
where $R_{n,j}$ converges in $P_{(n)}$-probability to $0$. Let $j\in\{1,2\}$ be fixed. Define $Y_{n,i}^{(j)}=\varepsilon_{n,i}^{(j)}(\mathrm{ d } \mu_{n,i}^{(j)}/\,\mathrm{ d }P_{n,i}^{(j)}-1)$.  Note that by Taylor's Theorem $\log(1+x)=x-\frac{x^2}{2}+(2/3)x^3(1+y_x)^{-3}$ for $|y_x|\leq x$. Since $\max_{1\leq i \leq k_n} Y_{n,i}^{(j)}\to 0$ in $P_{(n)}$-probability, see \Cref{theo:normal_limits}, it remains to shown that $\sum_{i=1}^{k_n} ( Y_{n,i}^{(j)})^2 \to \gamma(\boldsymbol{\theta_j},\boldsymbol{\theta_j})$ in $P_{(n)}$-probability. It is well know that this follows immediately if the Lindeberg condition is fulfilled for the triangular array $(Y_{n,i}^{(j)})_{i\leq k_n}$ under $P_{(n)}$.  Observe that combining \Cref{theo:normal_limits}\eqref{enu:theo:normal_limits_Zn} and the assumption $\gamma(\boldsymbol{\theta_j},\boldsymbol{\theta_j})=\sigma_j^2$ yields the desired Lindeberg condition and, finally, \eqref{eqn:LAN}.\\
From \eqref{eqn:LAN} and the asymptotic normality of the vector $(Z_n(\boldsymbol{\theta_1}),Z_n(\boldsymbol{\theta_2}))$ we obtain $(T_n(\boldsymbol{\theta_1}),T_n(\boldsymbol{\theta_2}))\overset{\mathrm d}{\longrightarrow} \widehat Z\sim  N ( (- \gamma(\boldsymbol{\theta_j},\boldsymbol{\theta_j})/2)_{j=1,2}, (\gamma(\boldsymbol{\theta_i},\boldsymbol{\theta_j}))_{1\leq i,j\leq 2} )$. Consequently, by the third lemma of Le Cam we get under $Q_{(n)}(\boldsymbol{\theta_1})$
\begin{align*}
(T_n(\boldsymbol{\theta_1}),T_n(\boldsymbol{\theta_2}))\overset{\mathrm d}{\longrightarrow} \widehat Z\sim  N \Bigl( \Bigl(\frac{- \gamma(\boldsymbol{\theta_j},\boldsymbol{\theta_j})}{2}+\gamma(\boldsymbol{\theta_1},\boldsymbol{\theta_j})\Bigr)_{j=1,2}, (\gamma(\boldsymbol{\theta_i},\boldsymbol{\theta_j}))_{1\leq i,j\leq 2} \Bigr).
\end{align*}
Finally, the desired statement can be concluded.

\subsubsection{Proof of \Cref{cor:violation}}
Define 
\begin{align*}
\varepsilon^*_{n,i}=\frac{\varepsilon_{n,i}(1-\kappa_{n,i})}{1-\varepsilon_{n,i}\kappa_{n,i}}.
\end{align*}
Now, let $ Q_{n,i}^*$, $Q_{(n)}^*$ and $ T_n^*$ defined as $\widetilde Q_{n,i}$, $ \widetilde Q_{(n)}$ and $\widetilde T_n$ replacing  $\widetilde\varepsilon_{n,i}$ by $\varepsilon_{n,i}^*$. 
Since $\widetilde \varepsilon_{n,i}= \varepsilon_{n,i}^*(1+a_{n,i})$ with $\max_{1\leq i \leq k_n}|a_{n,i}|\to 0$ it can easily be seen by \Cref{theo:xi_real_oder_xi_fullinfo,theo:general_limit_theorem}  that \eqref{eqn:def_Tn} also holds for $T_n^*$ with the same limits $\widetilde \xi_1$ and $\widetilde \xi_2$. Note that
\begin{alignat*}{2}
&\frac{\mathrm{ d } Q_{n,i}}{\,\mathrm{ d }P_{n,i}}
= (1-\varepsilon_{n,i}\kappa_{n,i}) \frac{\mathrm{ d } Q_{n,i}^*}{\,\mathrm{ d }P_{n,i}} + \infty\: \mathbf{1}_{N_{n,i}}\mathbf{1}\{\kappa_{n,i}>0\}.
\end{alignat*}
Since $\bigotimes_{i=1}^{k_n} N_{n,i}$ is a $P_{(n)}$-null set we obtain that $P_{(n)}$-almost surely
\begin{align*}
\log \Bigl( \frac{\mathrm{ d } Q_{(n)}}{\,\mathrm{ d }P_{(n)}} \Bigr) &=  \log \Bigl( \frac{\mathrm{ d } Q_{(n)}^*}{\,\mathrm{ d }P_{(n)}}  \Bigr) +  \sum_{i=1}^{k_n}\log(1-\kappa_{n,i}\varepsilon_{n,i}).
\end{align*}
Combining this and $\sum_{i=1}^{k_n}\log(1-\kappa_{n,i}\varepsilon_{n,i})\to -c$ yields that $T_n\overset{\mathrm d}{\longrightarrow} \widetilde\xi_1-c= \xi_1$ under $P_{(n)}$. By \Cref{sec:binary_exp_+_dist} we obtain that $T_n$ converges in distribution to some $\xi_2$  under $Q_{(n)}$ and by \eqref{eqn:relation_nu_2_nu_1} we get the desired representation $\nu_2=e^{-c}\widetilde\nu_2*\epsilon_{-c}+(1-e^{-c})\epsilon_\infty$ of $\xi_2$'s distribution
.
\subsection{Proofs of \Cref{sec:HCT}}
To shorten the notation we define
\begin{align*}
Z_n(t)= \sqrt{n}\;\frac{\mathbb{F}_n(t)-t}{\sqrt{t(1-t)}}, \,t\in(0,1).
\end{align*}
Then,
\begin{align*}
HC_n= \sup_{t\in(0,1)}|Z_n(t)|.
\end{align*}
\subsubsection{Proof of \Cref{theo:HC_full_power}}
First, note that
\begin{align*}
a_{n} HC_{n} - b_{n} = \sqrt{2}\, \log \log\left(  k_n \right)\,  \left( \frac{ HC_n }{ \sqrt{\log \log( k_n)}} - \sqrt{2} + o(1) \right).
\end{align*} 
That is why it sufficient to show that for some $\gamma>0$ 
\begin{align}
&Q_{(n)}\Bigl( \frac{ | Z_n(v_n) |}{\sqrt{\log\log k_n}} \leq \sqrt{2}+\gamma  \Bigr)\to 0 \label{eqn:HC_alt_vollt_proof_suff_v}\\
\textrm{ or }&Q_{(n)}\Bigl( \frac{ | Z_n(1-v_n) |}{\sqrt{\log\log k_n}} \leq \sqrt{2}+\gamma  \Bigr)\to 0 .\label{eqn:HC_alt_vollt_proof_suff_1-v}
\end{align}
To verify this we apply Chebyshev's inequality. Note that for every real-valued random variable $Z$ on some probability space $(\Omega,{\mathcal A },P)$ with finite expectation we have
\begin{align}\label{eqn:HC_proof_Cheby}
P\left(  \left| Z \right| \leq \frac{|E(Z)|}{2}\right) = P\left(  \left| Z - E(Z) \right| \geq \frac{\left| E(Z) \right|}{2}   \right)\leq 4\frac{\text{Var}_P(Z)}{E_P(Z)^2}.
\end{align}
Consequently, we need to determine first the expectation and variance for $Z_n(v)$ for $v\in\{v_n,1-v_n\}$:
\begin{align*}
E_{Q_{(n)}}( Z_n(v) )  &= \sqrt{k_n}\; \frac{ k_n^{-1} \sum_{i=1}^{k_n} Q_{n,i} (0,v]  -v }{ \sqrt{v(1-v) }} =  \frac{ \sum_{i=1}^{k_n} \varepsilon_{n,i} \left( \mu_{n,i} \left(0,v\right]  - v \right)}{\sqrt{ k_n\,v (1-v)}}\;,
\\[0.2em]
\text{Var}_{Q_{(n)}}( Z_{n}( v ) ) &=\; \frac{1}{k_n}\; \frac{ \sum_{i=1}^{k_n} Q_{n,i} (0,v] \left( 1- Q_{n,i} (0,v] \right)   }{ v(1-v) } \nonumber \\
&\leq   \min \Bigl\{ \frac{ \sum_{i=1}^{k_n} Q_{n,i}(0,v]}{ k_n\,v(1-v) }\;,\; \frac{ \sum_{i=1}^{k_n} (1-Q_{n,i}(0,v])}{ k_n\,v(1-v) }\Bigr\} \\
&= \min \Bigl\{ \frac{1}{1-v} + \frac{E_{Q_{(n)}}\left[ Z_n\left( v \right) \right]}{\sqrt{ k_n\,v(1-v) }}\;,\;\frac{1}{v}-\frac{E_{Q_{(n)}}\left[ Z_n\left( v \right) \right]}{\sqrt{ k_n\,v(1-v) }} \Bigr\} .
\end{align*}
By assumption we have
\begin{align}
&\frac{  | \sum_{i=1}^{k_n} \varepsilon_{n,i} (\mu_{n,i} (0,v_n] - v_n ) | }{\sqrt{k_nv_n\log\log(k_n)}}\to \infty  \label{eqn:HC_fullp_proof_Hn(vn)_to_infty_1}\\
\textrm{ or } \,\,&  \frac{ | \sum_{i=1}^{k_n} \varepsilon_{n,i} (\mu_{n,i} (1-v_n,1) - v_n ) |}{\sqrt{k_nv_n\log\log k_n}} \to \infty.\label{eqn:HC_fullp_proof_Hn(vn)_to_infty_2}
\end{align}
Suppose that \eqref{eqn:HC_fullp_proof_Hn(vn)_to_infty_1} holds. Then 
\begin{align*}
\left | \frac{E_{Q_{(n)}}( Z_n(v_n) ) }{\sqrt{\log\log(k_n)}} \right|\to\infty \textrm{ and }\frac{\text{Var}_{Q_{(n)}}( Z_n(v_n)) }{E_{Q_{(n)}}( Z_n(v_n) )^2 }\to 0.
\end{align*}
Combining this and \eqref{eqn:HC_proof_Cheby} yields that \eqref{eqn:HC_alt_vollt_proof_suff_v} is fulfilled for all $\gamma>0$. Analogously, if \eqref{eqn:HC_fullp_proof_Hn(vn)_to_infty_2} is true then \eqref{eqn:HC_alt_vollt_proof_suff_1-v} holds for all $\gamma>0$.

\subsubsection{Proof of \Cref{theo:HC_undetect}}
Let $G_{n}$ be the distribution function of $Q_{n,1}$, i.e. $G_{n}(v)=Q_{n,1}([0,v])$, $v\in(0,1)$. 
Let $U_1,U_2,\ldots$ be a sequence of independent, uniformly on $(0,1)$ distributed random variables on the same probability space $(\Omega,{\mathcal A }, P)$. Note $(U_1,\ldots,U_{k_n})\sim P_{(n)}$ and $(G_{n}^{-1}(U_1),\ldots,G_{n}^{-1}(U_{k_n}))\sim Q_{(n)}$, where $G_{n}^{-1}$ denotes the left continuous quantile function of $Q_{n,1}$. Moreover, denote the interval $(r_n,s_n)\cup(t_n,u_n)$ by $J_{n,1}$ and $[1-u_n,1-t_n]\cup [1-s_n,1-r_n]$ by $J_{n,2}$. By \eqref{eqn:HC_undect_condition} it is easy to see that we can replace $r_n$ by any $r_n'\geq r_n$ such that $\log(r_n')=( -1+o(1))\log(n)$. In particular, we can assume without loss of generality that $k_nr_n\geq 1$ and, analogously, $u_n<1/2$.
From Corollaries 2 and 3 as well as (1) and (2) of Theorem of \citet{Jaeschke1979}, which also hold for the statistics $W_n,\widehat V_n, \widehat W_n$ introduced at the beginning of subsection 2 therein, we can deduce that 
\begin{align}
& a_n \sup_{ v\in (0,1)\setminus (J_{n,1}\cup J_{n,2})}\Bigl\{ \;\Bigl | \frac{ \sum_{i=1}^{k_n}(\mathbf{1}\{U_i\leq v\}-v)}{\sqrt{k_nv(1-v)}} \Bigr| \Bigr\}   - b_n \overset{\mathrm P}{\longrightarrow} -\infty \label{eqn:HC_undetect_HC_to_-infty}\\
\textrm{ and }\,& a_n \sup_{ v\in (0,1)}\Bigl\{ \;\Bigl | \frac{ \sum_{i=1}^{k_n}(\mathbf{1}\{U_i\leq v\}-v)}{\sqrt{k_nv(1-v)}} \Bigr| \Bigr\}   - b_n  \overset{\mathrm d}{\longrightarrow} Y,\label{eqn:HC_undetect_HC_to_Y}
\end{align}
where the distribution function of $Y$ equals $\Lambda^2$, see \eqref{eqn:HC_limit_distr_null}. By \eqref{eqn:HC_undetect_HC_to_-infty}, the mutually contiguity of $P_{(n)}$ and $Q_{(n)}$ and the equivalence "$G_{n}(v)\geq u \Leftrightarrow v \geq G_{n}^{-1}(u)$" it is sufficient for  \eqref{eqn:HC_undetect} to verify
\begin{align}\label{eqn:HC_undetect_suff}
a_n \sup_{ v\in J_{n,1}\cup J_{n,2}}\Bigl\{ \;\frac{ \sum_{i=1}^{k_n}(\mathbf{1}\{U_i\leq G_n(v)\}-v)}{\sqrt{k_nv(1-v)}} \Bigr\}   - b_n  \overset{\mathrm d}{\longrightarrow} Y.
\end{align}
For this purpose we define 
\begin{alignat*}{2} 
&\Delta_{n,1}(v)= \frac{  \sum_{i=1}^{k_n}\left( \mathbf{1}\{U_i\leq {G}_{n}(v)\} - {G}_{n}(v) \right) }{ \sqrt{n\,{G}_{n}(v)(1-{G}_{n}(v))}},\\
&\Delta_{n,2}(v)=  \sqrt{ \frac{ {G}_{n}(v) }{ v } },\; 
\Delta_{n,3}(v)= \sqrt{  \frac{ 1-{G}_{n}(v) }{ (1-v) } },\; 
\Delta_{n,4}(v)= \sqrt{k_n}\, \frac{ {G}_{n}(v) -v  }{ \sqrt{\,v(1-v)}}.  
\end{alignat*} 
Clearly, 
\begin{align*}
\frac{  \sum_{i=1}^{k_n}\left( \mathbf{1}\{U_i\leq {G}_{n}(v )\} - v \right) }{ \sqrt{k_n\,v(1-v)} }\;=\; \Delta_{n,1}(v ) \Delta_{n,2}(v) \Delta_{n,3}(v)+ \Delta_{n,4}(v).
\end{align*}
Hence, the proof of \eqref{eqn:HC_undetect_suff} falls naturally into the following steps:	
\begin{align}
&\sup_{v \in J_{n,1}\cup J_{n,2}}  \left| \Delta_{n,j}(v) - 1 \right| \to 0 \textrm{ for }j\in\{2,3\},  \label{eqn:HC_alt_nichttr_Delta23_zuzeigen} \\
&  a_{n}\sup_{v\in J_{n,1}\cup J_{n,2}}  \left| \Delta_{n,4}(v) \right|  \to 0, \label{eqn:HC_alt_nichttr_Delta4_zuzeigen}\\
& \  a_{n} \sup_{v\in J_{n,1}\cup J_{n,2}}\{    | \Delta_{n,1}(v) | \} - b_{n} \overset{\mathrm d}{\longrightarrow} Y.\label{eqn:HC_alt_nichttr_Delta1_zuzeigen_j=1}
\end{align} 
First, observe that $(1-\varepsilon_{n})v \leq G_n(v) \leq v + \varepsilon_{n}(1-v)$ for all $v\in(0,1)$.
Hence, we have for all $v_1\in(0,1/2]$ and $v_2\in[1/2,1)$ that 
\begin{align*}
\frac{1-G_n(v_1)}{1-v_1},\, \frac{G_n(v_2)}{v_2}\in(1-\varepsilon_{n},  1+\varepsilon_{n}).
\end{align*}
Moreover, we have for all $v_1\in J_{n,1}$ and all $v_2\in J_{n,2}$ that
\begin{align}\label{eqn:HC_undetect_bound_2}
\left| \frac{ {G}_{n}\left(v_1\right)  }{v_1} -1 \right|  
&=  \frac{\varepsilon_{n} | \mu_{n} (0,v_1]  - v_1  | }{v_1}   \leq  \frac{H_n(v_1)}{  \sqrt{ k_n r_{n} }} \leq a_nH_n(v_1),\\
\label{eqn:HC_undetect_bound_3}
\left| \frac{ 1-{G}_{n}\left(v_2\right) }{1-v_2}-1  \right|   &
=   \frac{\varepsilon_{n} | \mu_{n} (v_2,1)  - (1-v_2)  |}{1-v_2} \,   \leq a_n H_n(1-v_2).
\end{align}
Consequently, \eqref{eqn:HC_alt_nichttr_Delta23_zuzeigen} follows. Similarly to the above, we obtain 
\begin{align*}
|\Delta_{n,4}(v_1) |\leq \frac{H_n(v_1)}{\sqrt{1-u_n}} \leq \frac{1}{\sqrt{2}} H_n(v_1) \textrm{ and }		|\Delta_{n,4}(v_2)|\leq \frac{1}{\sqrt{2}} H_n(1-v_2).
\end{align*}
for all $v_1\in J_{n,1}$ and $v_2\in J_{n,2}$. 
From this we obtain \eqref{eqn:HC_alt_nichttr_Delta4_zuzeigen}. Clearly,
\begin{align*}
\sup_{v \in J_{n,1}\cup J_{n,2}} |\Delta_{n,1}(v)| = \sup_{v \in \widetilde J_{n,1}\cup \widetilde J_{n,2}} \Bigl | \frac{ \sum_{i=1}^{k_n}(\mathbf{1}\{U_i\leq v\}-v)}{\sqrt{k_nv(1-v)}} \Bigr|,
\end{align*}
where $\widetilde J_{n,1}=[\widetilde r_n,\widetilde s_n]\cup[\widetilde t_n,\widetilde u_n]$ by  $\widehat J_{n,2}=(1-\widehat u_n,1-\widehat t_n)\cup (1-\widehat s_n,1-\widehat r_n)$ with $\widetilde{r}_n = G_n(r_n)$, $\widetilde{s}_n = G_n(s_n)$, $\widetilde{t}_n = G_n(t_n)$, $\widetilde{u}_n = G_n(u_n)$, $\widehat{r}_n = 1-G_n(1-r_n)$, $\widehat{s}_n = 1-G_n(1-s_n)$, $\widehat{t}_n = 1-G_n(1-t_n)$ and $\widehat{u}_n = 1-G_n(1-u_n)$. From \eqref{eqn:HC_undetect_bound_2}, \eqref{eqn:HC_undetect_bound_3} and \eqref{eqn:HC_undect_condition}  we deduce that ($\widetilde r_n$,$\widetilde s_n$,$\widetilde t_n$,$\widetilde u_n$) and ($\widehat r_n$,$\widehat s_n$,$\widehat t_n$,$\widehat u_n$) fulfil \eqref{eqn:HC_cond_rn_un_sn_tn}. Finally, \eqref{eqn:HC_alt_nichttr_Delta1_zuzeigen_j=1}  follows from  \eqref{eqn:HC_undetect_HC_to_-infty} and \eqref{eqn:HC_undetect_HC_to_Y} (with the new parameters).

\subsection{Proofs of \Cref{sec:h-model}}
Before we prove the theorems stated in \Cref{sec:h-model} we want to point out the following: We can always assume that there is no perturbation, i.e. $r_{n,i}=0$ for all $i,n$, see Lemma \ref{lem:h-model_rni}. Note that we will assume this in all upcoming proofs concerning \Cref{sec:h-model} without recalling it every time.
\begin{lemma}[Perturbation]\label{lem:h-model_rni}
	Let us consider the situation in \Cref{sec:h-model}. Let $\mu_{n,i}^*$, $Q_{n,i}^*$ and $Q_{(n)}^*$ be defined as $\mu_{n,i}$, $Q_{n,i}$ and $Q_{(n)}$ setting $r_{n,i}=0$ for all $i,n$. Then \eqref{lem:h-model_rni_condition} is a sufficient that $\{Q_{(n)},Q_{(n)}^*\}$ converges weakly to the uninformative experiment $\{\epsilon_0,\epsilon_0\}$. In other words, if \eqref{lem:h-model_rni_condition} is fulfilled then the perturbation by $(r_{n,i})_{i\leq k_n}$ does not affect the asymptotic results.
\end{lemma}
\begin{proof}
	It is sufficient to show that $\sum_{i=1}^{k_n}d^2(Q_{n,i},Q_{n,i}(\boldsymbol{\widetilde\theta})) \to 0$ under \eqref{lem:h-model_rni_condition}. This convergence follows immediately from the first representation of the Hellinger distance in \eqref{eqn:def_hellinger} and the third binomial formula:
	\begin{align*}
	&2d^2(Q_{n,i},Q_{n,i}(\boldsymbol{\widetilde\theta}))\\
	&= \int_0^1 \frac{ (\varepsilon_{n,i}r_{n,i})^2}{ (\sqrt{(1-\varepsilon_{n,i})+\varepsilon_{n,i}h_{n,i}}+\sqrt{(1-\varepsilon_{n,i})+\varepsilon_{n,i}h_{n,i}+\varepsilon_{n,i}r_{n,i}})^2} \,\mathrm{ d }\lebesgue \\
	&\leq \frac{ \varepsilon_{n,i}^2}{1-\varepsilon_{n,i}} \int_0^1r_{n,i}^2\,\mathrm{ d }\lebesgue.
	\end{align*}
\end{proof}

\subsubsection{Proof of \Cref{theo:h-model}}
First, observe that
\begin{align}
& I_{n,1,x} = \sum_{i=1}^{k_n} \varepsilon_{n,i} \int_0^1 h_{n,i}\, \mathbf{1}\Bigr\{ \frac{ \varepsilon_{n,i} }{ \kappa_{n,i}} h_{n,i} > x \Bigl\} \,\mathrm{ d } \lebesgue,	\label{eqn:h-model_I1}\\
& I_{n,2,x} = \sum_{i=1}^{k_n} \Bigl( \frac{\varepsilon_{n,i}^2}{\kappa_{n,i}} \int_0^1 h_{n,i}^2\, \mathbf{1}\Bigr\{ \frac{ \varepsilon_{n,i} }{ \kappa_{n,i}} h_{n,i} \leq  x \Bigl\} \,\mathrm{ d } \lebesgue  \Bigr)- \sum_{i=1}^{k_n}\varepsilon_{n,i}^2	\label{eqn:h-model_I2}.
\end{align}
Moreover, note that
\begin{align}
I_{n,1,x} &\leq \frac{1}{x}\sum_{i=1}^{k_n} \frac{\varepsilon_{n,i}^2}{\kappa_{n,i}} \int h_{n,i}^2\, \mathbf{1}\Bigr\{ \frac{ \varepsilon_{n,i} }{ \kappa_{n,i}} h_{n,i} >  x \Bigl\} \,\mathrm{ d } \lebesgue, \label{eqn:h-model_I1_bound} \\
I_{n,1,x}+I_{n,2,x}&\leq \max\{1,x^{-1}\} \Bigl( \max_{1\leq i \leq k_n} \int_0^1 h_{n,i}^2 \,\mathrm{ d }\lebesgue \Bigr) \sum_{i=1}^{k_n}\frac{\varepsilon_{n,i}^2}{\kappa_{n,i}}\nonumber 
\end{align}
By these and \Cref{theo:trivial_limits} $K=0$ corresponds to the undetectable case and no accumulation point of $\{P_{(n)},Q_{(n)}\}$ is full informative if $K\in(0,\infty)$. By Lemma \ref{lem:D_detection}\eqref{enu:lem:D_detection_comp} and \eqref{eqn:totvar_helg_absch_P+mu} the latter is also valid if $\limsup_{n\to\infty}\sum_{i=1}^{k_n}\varepsilon_{n,i}<\infty$. Consequently, \eqref{enu:theo:h-model_undetect} and the first statement in \eqref{enu:theo:h-model_nontrivial} are verified. Now, let us suppose that $K\in(0,\infty)$ and \eqref{eqn:h-model_max_condition} holds. Clearly, $\sum_{i=1}^{k_n}\varepsilon_{n,i}^2\to 0$. By \eqref{eqn:h-model_I2} and \eqref{eqn:h-model_I1_bound} $I_{n,1,x}\to 0$ and $ I_{n,2,x}\to K \int_0^1 h^2 \,\mathrm{ d }\lebesgue=\sigma^2$ for all $x>0$. Hence, applying \Cref{theo:general_limit_theorem} completes the proof of \eqref{enu:theo:h-model_nontrivial}.\\
Now, let the assumptions of \eqref{enu:theo:h-model_comp_detect} hold. Without loss of generality we can assume that $\sum_{i=1}^{k_n}\varepsilon_{n,i}^2\to C_1<\infty$ and $\varepsilon_{n,r_n} / \kappa_{n,r_n} \to C\in[0,\infty]$ since otherwise we use standard sub-subsequence arguments and make use of \eqref{eqn:totvar_helg_absch_P+mu}. If $C\geq 1$ then for all sufficiently large $n\in\N$
\begin{align*}
I_{n,1,x} &\geq  \sum_{i=r_n}^{k_n} \varepsilon_{n,i}  \int_0^1 h_{n,i}\, \mathbf{1}\Bigr\{ \frac{ \varepsilon_{n,r_n} }{ \kappa_{n,r_n}} h_{n,i} > x \Bigl\} \,\mathrm{ d } \lebesgue \\
& \geq \Bigl( \sum_{i=r_n}^{k_n} \varepsilon_{n,i}  \Bigr) \min_{1\leq i \leq k_n} \int_0^1 h_{n,i}\, \mathbf{1}\Bigr\{ \frac{1}{2} h_{n,i} > x \Bigl\} \,\mathrm{ d } \lebesgue
\end{align*}
and so by \eqref{eqn:h_unif_conv} $I_{n,1,x}\to \infty$ for all sufficiently small $x>0$. If $C< 1$ then
\begin{align*}
I_{n,2,x} &\geq  \Bigl( \sum_{i=1}^{r_n}  \frac{\varepsilon_{n,i}^2}{\kappa_{n,i}} \Bigr) \min_{1\leq i \leq k_n}\int_0^1 h_{n,i}^2\, \mathbf{1}\Bigr\{ 2 h_{n,i} \leq  x \Bigl\} \,\mathrm{ d } \lebesgue  - C_1
\end{align*}
and so by \eqref{eqn:h_unif_conv} $I_{n,2,x}\to \infty$ for all sufficiently large $x>0$. Hence, applying \Cref{theo:trivial_limits} verifies \eqref{enu:theo:h-model_comp_detect}. Finally, note that $K<\infty$ implies $\sum_{i=1}^{k_n}\varepsilon_{n,i}^2\to 0$. Keeping this in mind the proof of  \eqref{enu:theo:h-model_nonpara_LAN} is trivial (and omitted to the reader).
\subsubsection{Proof of \Cref{theo:h-model_beta=1}}
By \eqref{eqn:h-model_I2} 
\begin{align*}
& -\frac{1}{k_n}\leq I_{n,2,x} \leq \frac{x}{k_n} \sum_{i=1}^{k_n} \int_0^1 h_{n,i} \mathbf{1}\{ k_n^{r-1} h_{n,i}\leq x \} \,\mathrm{ d }\lebesgue \leq x	\\
&\textrm{and so }
\lim_{ x \searrow 0}\underset{n\to\infty}{ \substack{ \limsup \\ \liminf } } \,\,I_{n,2,x} = 0.
\end{align*}
Combining \eqref{eqn:h-model_I1} and \eqref{eqn:theo:h-model_beta=1_condi} yields for all $x\in {\mathcal D }$ that $I_{n,1,e^x-1}$ equals
\begin{align*}
& \frac{1}{k_n}\sum_{i=1}^{k_n} \int_0^1 h_{n,i}\, \mathbf{1}\Bigr\{ k_n^{r-1} h_{n,i} > e^x-1 \Bigl\} \,\mathrm{ d } \lebesgue\to \mathbf{1}\{r>1\} + M(x,\infty)\mathbf{1}\{r=1\}.	
\end{align*}
Consequently, applying \Cref{theo:general_limit_theorem} and \Cref{theo:trivial_limits} completes the proof.

\subsubsection{Proof of \Cref{lem:h=xrho}}
It is easy to verify that by \eqref{eqn:h-model_I1} and \eqref{eqn:h-model_I2}
\begin{align*}
&I_{n,1,x} = \min\Bigl\{ n^{1-\beta},n^{\frac{1}{\alpha}(\alpha-\beta + r(1-\alpha))} \Bigl( \frac{x}{1-\alpha} \Bigr)^{1-\frac{1}{\alpha}}\Bigr\},\;\, I_{n,2,x} \leq n^{1-2\beta+r}.
\end{align*}
Note that $1-2\beta+r<0$ if $r<\rho^*(\beta,\alpha)$, or if $r=\rho^*(\beta,\alpha)$ and $\alpha>1/2$. Moreover,   in the case of $\alpha=1/2$, $r=\rho^*(\beta,\alpha)= 2\beta-1$ we have 
\begin{align*}
I_{n,2,x} = \frac{1}{2}\log (2xn^{1-\beta}) - n^{1-2\beta} \to 0.
\end{align*}
Combining these, \Cref{theo:general_limit_theorem}, \Cref{theo:trivial_limits} and \eqref{eqn:conn_eta_M} completes the proof.

\subsubsection{Proof of \Cref{theo:HC_hmodel}}
To shorten the notation, set $\mu_n=\mu_{n,1}$, $\kappa_n=\kappa_{n,i}$ and $\varepsilon_{n}=\varepsilon_{n,i}$. Since the support of $\mu_n$ is $(0,\kappa_n)$ with $\kappa_n\to 0$ and, clearly, $a_nk_n\varepsilon_{n}^2=a_nk_n^{1-2\beta}\to 0$ we deduce from Remark \ref{rem:HC} that we can replace $H_n(v)$ in \Cref{theo:HC_full_power,theo:HC_undetect} by
\begin{align*}
\widehat H_n(v) = k_n^{\frac{1}{2}-\beta} v^{-\frac{1}{2}} \mu_n(0,v) = k_n^{\frac{1}{2}-\beta} v^{-\frac{1}{2}} \int_0^{\min\{vk_n^r,1\}} h \,\mathrm{ d }\lebesgue.
\end{align*}
We give the proof for the model  \eqref{enu:theo:HC_hmodel_intro} and the one from \Cref{theo:h-model_beta=1} in the case of $r=1$. The  model \eqref{enu:theo:HC_hmodel_h=xrho} is much simpler and left to the reader. 
\\
First, consider $\beta=r=1$. Let $r_n=k_n^{-1}a_n^3$, $s_n=t_n$ and $u_n=(\log k_n)^{-1}$. Clearly, \eqref{eqn:HC_cond_rn_un_sn_tn} holds. Moreover,
\begin{align*}
a_n\sup\{ \widehat H_n(v): v \in [r_n,u_n]\} \leq a_n k_n^{\frac{1}{2}-\beta} r_n^{-\frac{1}{2}}\to 0.
\end{align*}
Hence, by \Cref{theo:HC_undetect} the HC test has no power asymptotically.
\\
Now, consider the model from \Cref{sec:intro_h}  with $h\in L^{2+\delta}(\lebesgue_{|(0,1)})$ for some $\delta\in(0,1)$. In particular, we have $k_n=n$. First, let $r>\rho(\beta)=1-2\beta$ and $\beta<1$. Set $v_n=n^{-\min\{1,r\}}$. Clearly, $n^rv_n\geq 1$ and
\begin{align*}
a_n^{-1}\widehat H_n(v_n) = a_n^{-1} n^{1/2-\beta+\min\{1,r\}/2}  \to \infty.
\end{align*}
By this, \Cref{theo:HC_full_power,theo:h-model} the areas of complete detection ($r>\rho(\beta)$) coincide for the HC and the LLR test. 
It remains to discuss $r=\rho(\beta)=2\beta-1$ and $\beta<1$. Set $r_n=n^{-1}$, $s_n=n^{-r}a_n^{-4(1+2/\delta)}$,  $t_n=n^{-r}a_n^{4}$ and $u_n=(\log n)^{-1}.$	
Clearly, \eqref{eqn:HC_cond_rn_un_sn_tn} holds. By H\"older's inequality there is some $c_0>0$ such that 
\begin{align*}
\mu_n(0,v] \leq \left( \int_0^1 h^{2+\delta} \,\mathrm{ d } \lebesgue \right)^{1/(2+\delta)} \,\left( \int_0^{vn^r} \,\mathrm{ d } \lebesgue\right)^{1-1/(2+\delta)} \leq c_0\, (vn^r)^{1 -1/(2+\delta)}
\end{align*} 
for all $v\in(0,1)$. Hence, we obtain
\begin{align*}
a_n\sqrt{n}\varepsilon_{n} \sup_{v\in[r_n,s_n ]}\left\{ \frac{\mu_n (0,v]}{\sqrt{v}} \right\}
\leq  a_n\,n^{1/2-\beta} c_0   s_n^{1/2-1/(2+\delta)}\, n^{r -r/(2+\delta) } \leq c_0\,a_n^{-1}  \to 0.
\end{align*}
Moreover,
\begin{align*}
a_n\sqrt{n}\varepsilon_n \sup_{v\in[t_n,u_n]} \left\{ \frac{\mu_n (0,v]}{\sqrt{v}} \right\}\; \leq \; a_n n^{1/2-\beta}\,t_n^{-1/2}\; =\;  a_n^{-1}\to 0.
\end{align*}
Finally, by \Cref{theo:HC_undetect} the HC test has no power asymptotically.

\subsection{Proofs for \Cref{sec:normal_mix}}

\subsubsection{Proof of \Cref{theo:heterosce_normal_HCT}}
First, remind that we apply the HC statistic to $p_{n,i}=1-\Phi(Y_{n,i})$. Hence, without loss of generality we can write $\mu_n = N(\vartheta_n,\sigma_0^2)^{1-\Phi}$. Note that
\begin{align}\label{eqn:HC_normal_mu(0,v)}
\mu_n(0,v] = 1 - \Phi\Bigl( -\frac{ \Phi^{-1}(v)+\vartheta_n }{ \sigma_0 } \Bigr), \,v\in(0,1).
\end{align}
Moreover, we have for all $v\in(0,1/2)$
\begin{align}\label{eqn:HC_normal_mu(1-v,1)}
\mu_n(1-v,1] = 1- \Phi \Bigl( \frac{-\Phi^{-1}(v)+\vartheta_n}{\sigma_0} \Bigr) \leq \mu_n(0,v].
\end{align}
Observe that by Remark \ref{rem:contiguous} and \Cref{theo:detectbou_norm_sparse} $P_{(n)}$ and $Q_{(n)}$ are mutually contiguous. Clearly, this is not affected by the transformation to $p$-values.
Consequently, by \eqref{eqn:HC_normal_mu(1-v,1)}, \Cref{theo:HC_undetect} and Remark \ref{rem:HC} it is sufficient to show that
\begin{align*}
&a_n \sqrt{n}\varepsilon_{n} \sup_{v\in (n^{-1+\lambda_n},1/2]}\frac{\mu_n(0,v]}{\sqrt{v} }\to 0 \text{ with }\lambda_n=\frac{(\log\log(n))^2}{\log(n)},
\end{align*}
i.e. $r_n=n^{-1+\lambda_n}$, $s_n=t_n$ and $u_n=1/2$. Let $\delta>0$ be sufficiently small that $2\delta< 1- r\text{ and }2\delta\leq \beta-1/2-r/2$, where $2\beta -1-r$ is positive. Then
\begin{align*}
&a_n \sqrt{n}\varepsilon_{n}\sup_{v\in (n^{-r-2\delta},1/2]} \Bigl\{ \frac{\mu_n(0,v]}{\sqrt{v}}\Bigr\} \leq a_n (\log(n))^{E(\beta,\sigma_0)} n^{1/2-\beta+r/2+\delta} \to 0.
\end{align*}
Consequently, by \Cref{theo:HC_undetect} it remains to show that
\begin{align*}
a_n n^{1/2-\beta} (\log(n))^{E(\beta,\sigma_0)} \sup_{\kappa\in [r+2\delta,1-\lambda_n]} n^{\kappa/2}\mu_n(0,n^{-\kappa}] \to 0.
\end{align*}
For this purpose, a fine analysis of the tail behaviour of $\Phi$ is required. 
\begin{lemma}\label{lem:abschae_survivalfkt_standardnormal}
	We have
	\begin{align}\label{eqn:schachtelung_surv_durch_dichte_von_N(0,1)}
	\frac{x}{\sqrt{2\pi}(1+x^2) }\exp\Bigl(-\frac{1}{2}x^2\Bigr) \leq 1-\Phi(x)   \leq \frac{1}{\sqrt{2\pi}x}\exp\Bigl(-\frac{1}{2}x^2\Bigr)\: 
	\end{align}
	for all $x>0$. Moreover, there is some $U>0$ such that for all $u\in(0,U)$
	\begin{align}
	-\Phi^{-1}(u)= \Phi^{-1}(1-u) \geq \sqrt{2\log(u^{-1})} \Bigl( 1- \frac{7+\log\log(u^{-1})}{4\log(u^{-1})} \Bigr)\label{eqn:quantil_N(0,1)_PSI}.
	\end{align}
\end{lemma}
\begin{proof}
	From integration by parts we obtain for all $x>0$
	\begin{align*}
	1-\Phi(x) = \int_x^\infty \frac{1}{\sqrt{2\pi}} \frac{1}{t}\; t e^{-t^2/2}\,\mathrm{ d }t= \frac{1}{x\sqrt{2\pi}} e^{-x^2/2} - \int_x^\infty \frac{1}{t^2\sqrt{2\pi}}e^{-t^2/2} \,\mathrm{ d }t.
	\end{align*}
	Hence, the upper bound in \eqref{eqn:schachtelung_surv_durch_dichte_von_N(0,1)} follows. Since the integral on the right-hand side is smaller than $x^{-2}(1-\Phi(x))$ also the lower bound follows. Clearly, $\Phi^{-1}$ is increasing and $\Phi^{-1}(1-u)\to\infty$ as $u\searrow 0$. Let $U>0$ such that $\Phi^{-1}(1-U) >1$. By applying \eqref{eqn:schachtelung_surv_durch_dichte_von_N(0,1)} for $x=\Phi^{-1}(1-u)$ with $u\in(0,U)$ 
	\begin{align}
	\Phi^{-1}(1-u)  \leq \sqrt{ -2 \log ( u\, \sqrt{2\pi}\Phi^{-1}(1-u) ) } \leq \sqrt{-2\log(u)}.	\label{eqn:quantil_N(0,1)_absc_bew_leq_grob}
	\end{align}
	Obviously, by \eqref{eqn:schachtelung_surv_durch_dichte_von_N(0,1)} we have $(1/6x)\exp(-x^2/2) \leq 1-\Phi(x)$ for all $x>1$. By setting again $x=\Phi^{-1}(1-u)$ for $u\in(0,U)$ we obtain from this, \eqref{eqn:quantil_N(0,1)_absc_bew_leq_grob} and $\sqrt{1-y}\geq 1-y/2-y^2$ for all $y\in(0,1)$ that
	\begin{align*}
	\Phi^{-1}(1-u) 
	&\geq \sqrt{2\log(u^{-1})}\sqrt{1-\frac{\log(6)+\log(2)/2+\log\log(u^{-1})/2}{\log(u^{-1})}}\\
	&\geq \sqrt{2\log(u^{-1})} \Bigl( 1- \frac{ 3+ \log\log (u^{-1})/2 }{ 2\log(u^{-1})}-\Bigl( \frac{ 3+ \log\log (u^{-1})/2 }{ \log(u^{-1})} \Bigr)^2 \Bigr).
	\end{align*}
	Finally, by choosing $U>0$ sufficiently small we get \eqref{eqn:quantil_N(0,1)_PSI}.
\end{proof}
From now on, let $n\in\N$ be sufficiently large such that $n^{-1+\lambda_n}< U$ and so \eqref{eqn:quantil_N(0,1)_PSI} holds for all $u=n^{-\kappa}$, $\kappa\leq 1-\lambda_n$. We obtain for all $\kappa\in[r+2\delta,1-\lambda_n]$
\begin{align*}
- \Phi^{-1}(n^{-\kappa})-\vartheta_n \geq \sqrt{ 2 \log(n) } \Bigl( \sqrt{\kappa}-\sqrt{r} - \frac{\log(\kappa)+\log\log (n)+7}{4\sqrt{\kappa}\log(n)} \Bigr)=:w_n(\kappa).
\end{align*}
Hence, by \eqref{eqn:HC_normal_mu(0,v)} and \eqref{eqn:schachtelung_surv_durch_dichte_von_N(0,1)} there is  $c>0$ such that for all $\kappa\in[r+2\delta,1-\lambda_n]$  
\begin{align*}
n^{\frac{1}{2}\kappa}\mu_n(0,n^{-\kappa}] &\leq n^{\frac{1}{2}\kappa} \Bigl(1-\Phi\Bigl( \frac{w_n(\kappa)}{\sigma_0} \Bigr) \Bigr) \leq n^{\frac{1}{2}\kappa}\frac{ \sigma_0}{ w_n(\kappa)} \exp\Bigl(-\frac{1}{2\sigma_0^2}w_n(\kappa)^2\Bigr)\\
& \leq c n^{E_1(\kappa)} (\log(n))^{E_2(\kappa)} \text{ with }E_2(\kappa)=-\frac{1}{2}+\frac{1}{2}\frac{\sqrt{\kappa}-\sqrt{r}}{\sigma_0^2\sqrt{\kappa}}\\
&\text{and }E_1(\kappa)=\frac{1}{2}\kappa + \sigma_0^{-2}(2\sqrt{\kappa r}-\kappa - r).
\end{align*} 
Since we are interested in the supremum of all $\kappa\in[r+2\delta,1-\lambda_n]$ we need to find the (uniquely) point $\kappa_n^*\in[r+2\delta,1-\lambda_n]$ attaining the maximum of $[r+2\delta,1-\lambda_n]\ni \kappa \to E_1(\kappa)$. For this purpose we need to discuss two cases.

First, let $\sigma_0<\sqrt{2}$ and $r< (2-\sigma_0^2)^2/4$ (or equivalently $\beta<1-\sigma_0^2/4$). Then $E(\beta,\sigma_0)=0$, $\varepsilon_{n}=n^{-\beta}$ and $r=(2-\sigma_0^2)( \beta-1/2 )$. Without loss of generality we  assume that $r+2\delta < 4r  (2-\sigma_0^2)^{-2} < (1-\delta)^2 \textrm{ and }\delta (2-\sigma_0^2)/(4\sigma_0^2)<1/8.$
Then it is easy to verify that	${\kappa_n^*}=\kappa^*=4r  /(2-\sigma_0^2)^2 $ and $E_{1}(\kappa_n^*) =  r/(2-\sigma_0^2)$. Since $E_2$ is increasing we have for all sufficiently large $n\in\N$ that 
\begin{align*}
a_n\sqrt{n} \varepsilon_{n}\sup_{\kappa\in [r+2\delta,1-\lambda_n]} n^{\kappa/2}\mu_n(0,n^{-\kappa}] &=a_n\sup_{\kappa\in [r+2\delta,\kappa^*(1-\delta)^{-2}]} n^{\kappa/2+1/2-\beta}\mu_n(0,n^{-\kappa}] \\
&\leq  a_nc\, n^{E_1(\kappa^*)+1/2-\beta} (\log(n))^{E_2(\kappa^*(1-\delta)^{-2})}\\
&\leq a_n c\; (\log(n))^{-1/8} \to 0.
\end{align*}

Second, let $(\beta,\sigma_0)\in( 1-1/\sigma_0^2,1)\times(\sqrt{2},\infty)$ or $(\beta,\sigma_0)\in[ 1-\sigma_0^2/4,1)\times(0,\sqrt{2})$. 
Clearly, $E_1$ and $E_2$ are increasing in $[r+2\delta,1]$. Hence, $\kappa_n^*=1-\lambda_n$. Since $r=( 1-\sigma_0 \sqrt{1-\beta})^2$, $1/2-1/{\sigma_0^2}+2\sqrt{r}/\sigma_0^2-r/\sigma_0^2=\beta-1/2$ and $\sqrt{1-\lambda_n}\leq 1-\lambda_n/2$ we obtain that 
\begin{align*}
E_1(1-\lambda_n)&= \beta-\frac{1}{2} + \lambda_n \Bigl( \frac{1}{\sigma_0^2}-\frac{1}{2} \Bigr)+ \frac{2}{\sigma_0^2}\sqrt{r} (\sqrt{1-\lambda_n}-1)\\
&\leq \beta-\frac{1}{2} - K(\beta,\sigma_0^2)\lambda_n,\text{ where } \\
&K(\beta,\sigma_0^2)= \frac{1}{2}-\frac{1}{\sigma_0}\sqrt{1-\beta} 
\begin{cases}
= 0 &\text{if }\beta=1-\frac{1}{4}\sigma_0^2,\,\sigma_0<\sqrt{2}. \\
> 0& \text{else}.
\end{cases}
\end{align*}
Moreover, $E_2(1)=-\frac{1}{4}<0$ if $\beta=1-\sigma_0^2/4$, $\sigma^2_0<\sqrt{2}$. Consequently, 
\begin{align*}
&a_n \sqrt{n}\varepsilon_{n}\sup_{\kappa\in [r+2\delta,1-\lambda_n]} n^{\kappa/2}\mu_n(0,n^{-\kappa}] \\
&\leq  a_nc\, n^{E_1(1-\lambda_n)+1/2-\beta} (\log(n))^{E_2(1)+E(\beta,\sigma_0^2)}\\
&\leq a_n c(\log(n))^{E_2(1)+E(\beta,\sigma_0^2)-K(\beta,\sigma_0^2) \log\log(n)}  \to 0.
\end{align*}

\subsubsection{Proof of \Cref{theo:normal_ext_boundary}}
By careful calculations we obtain 
\begin{align*}
\frac1n \frac{\mathrm{ d } \mu_{n}}{\,\mathrm{ d }P_0}(x+\vartheta_n) = \frac{1}{\sigma_0}\exp \Bigl( \frac{\sigma_0^2-1}{ 2 \sigma_0^2 }x^2 + x\sqrt{2r\log n} + (r-1)\log n \Bigr).
\end{align*}
Define $C_{n,\tau}=\{x\in\R: n^{-1}\frac{\mathrm{ d } \mu_{n}}{\,\mathrm{ d }P_0}(x+\vartheta_n) > \tau\}$, $\tau>0$. It is easy to see that $\mathbf{1}\{x\in C_{n,\tau}\} \to  \mathbf{1}\{ r=1 , x>0\} + \mathbf{1}\{r>1\}$  for $x\neq 0$. From this and Lebesgue's dominated convergence theorem we deduce that
\begin{align*}
& I_{n,1,\tau}
=	\int \mathbf{1}\{x\in C_{n,\tau}\} \mathrm{ d }N(0,1)(x)
\to \mathbf{1}\{r>1\}-\frac{1}{2}\mathbf{1}\{r=1\}.
\end{align*}
Moreover, 
\begin{align*}
I_{n,2,\tau} &\leq \tau \int \frac{\mathrm{ d } \mu_{n}}{\,\mathrm{ d }P_0} \mathbf{1}\Bigl\{ \frac{1}{n} \frac{\mathrm{ d } \mu_{n}}{\,\mathrm{ d }P_0}\leq \tau \Bigr\}\,\mathrm{ d }P_0\leq \tau. 
\end{align*}
Finally, combining \Cref{theo:general_limit_theorem} and \Cref{theo:trivial_limits} yields the statement.

\section*{Acknowledgments}
The authors thanks the \textit{Deutsche Forschungsgemeinschaft} (DFG) for financial support (Grant no. 618886).

\end{document}